\documentclass[12pt]{amsart}

\usepackage[left=2.3cm,right=2.3cm, top=2.5cm, bottom=2.5cm]{geometry}
\usepackage{amscd}
\usepackage{amssymb}
\usepackage{mathrsfs}
\usepackage{amsmath}
\usepackage{mathabx}
\usepackage{amsthm}
\usepackage{tikz-cd}
\usepackage[all,cmtip]{xy}
\usepackage{enumerate}
\usepackage{enumitem}
\usepackage{tikz}
\usetikzlibrary{calc,decorations.pathreplacing}
\usepackage{mathtools}
\usepackage{color}
\usepackage{wasysym}
\usepackage{multirow}
\usepackage{comment}

\newcommand\bigw{\mathop{\raisebox{0.2ex}{\scalebox{.85}[0.85]{$\bigwedge$}}}\!}
\allowdisplaybreaks

\usepackage[colorlinks=true, linkcolor=blue, citecolor=blue, urlcolor=blue]{hyperref}
\numberwithin{equation}{section}

\newcommand{\thmref}[1]{Theorem~\ref{#1}}
\newcommand{\secref}[1]{Section~\ref{#1}}

\newcommand{\lemref}[1]{Lemma~\ref{#1}}
\newcommand{\propref}[1]{Proposition~\ref{#1}}
\newcommand{\corref}[1]{Corollary~\ref{#1}}

\newcommand{\figref}[1]{Figure~\ref{#1}}
\newcommand{\rmkref}[1]{Remark~\ref{#1}}

\newtheorem{theorem}{Theorem}[section]
\newtheorem{lemma}[theorem]{Lemma}
\newtheorem{proposition}[theorem]{Proposition}
\newtheorem{corollary}[theorem]{Corollary}

\theoremstyle{definition}

\newtheorem{notation}[theorem]{Notation}

\theoremstyle{remark}
\newtheorem{remark}[theorem]{Remark}

\newcommand{\fg}{\mathfrak{g}}

\newcommand{\gl}{\mathfrak{gl}}

\newcommand{\U}{{\rm U}}

\newcommand{\End}{{\rm End}}

\begin{document}

\title[Higher-order quantum Casimir elements]{Higher-order Casimir elements and hook partitions \\  for quantum groups of types $\mathsf{B}$, $\mathsf{C}$ and $\mathsf{D}$}
 \author[Yanmin Dai]{Yanmin Dai}
\address{\it Yanmin Dai: School of Mathematics and Statistics, Henan University of Science and Technology, Henan, 471000, China}
\email{yanmindai@haust.edu.cn}

\author[Yang Zhang]{Yang Zhang}
\address{\it Yang Zhang (corresponding author):  School of Mathematics and Physics, The University of Queensland, St Lucia, QLD 4072, Australia}
\email{yang.zhang@uq.edu.au}


\begin{abstract}

The higher-order quantum Casimir elements, introduced by Zhang, Bracken, and Gould in the early 1990s, were conjectured to generate the centre of the Drinfeld–Jimbo quantum (super)groups.
This was previously confirmed in the classical type $\mathsf{A}$. In this paper we extend the result to  the classical types $\mathsf{B}$, $\mathsf{C}$, $\mathsf{D}$, with the additional inclusion of quantum Casimir elements arising from spin and half-spin representations in types $\mathsf{B}$ and $\mathsf{D}$.  We identify the Harish–Chandra images of these elements and  reinterpret them uniformly  in terms of irreducible characters associated with hook partitions. This yields explicit and minimal generating sets for the centres in all classical types, and provides new connections between higher-order quantum Casimir elements and hook partitions that exhibit a stability phenomenon.
\end{abstract}
  
\maketitle

\vspace{-0.5cm}

\section{Introduction}

\subsection{Brief history} 
 Let $\mathrm{U}_q(\mathfrak{g})$ denote the Drinfeld-Jimbo quantum group \cite{Dri87,Jim85} associated with a finite-dimensional complex  simple Lie algebra $\mathfrak{g}$, where $q$ is an indeterminate. Understanding the structure of the centre $\mathcal{Z}(\mathrm{U}_q(\mathfrak{g}))$ is crucial for the representation theory of quantum groups and has important implications in various fields such as knot theory \cite{RT91}, statistical mechanics \cite{Bax82},  and conformal field theory \cite{DMS97}.

There are several explicit constructions of central elements of quantum groups in the literature.  For the quantum general linear group $\mathrm{U}_q(\mathfrak{gl}_n)$, Jimbo \cite{Jim86} obtained central elements using the fused $R$-matrix. Faddeev, Reshetikhin, and Takhtajan \cite{FRT90} described a generating set for the centre $\mathcal{Z}(\mathrm{U}_q(\mathfrak{gl}_n))$, although without including a proof. More generally, Drinfeld \cite{Dri90} constructed  central elements for arbitrary quasi-triangular Hopf algebras via the universal $R$-matrix. Based on this, Reshetikhin \cite{Res90} explained how invariants of certain tangles naturally yield central elements in a ribbon Hopf algebra. Applying Drinfeld’s universal $R$-matrix, Zhang, Bracken, and Gould \cite{ZGB91} constructed a family of central elements $C_{V,\ell}$ for quantum groups, associated with positive integer $\ell$ and  any finite-dimensional representation $V$.  These elements, which we refer to as \emph{higher-order quantum Casimir elements}, play a central role in what follows.

The algebraic structure of the centre  is described by a quantum analogue of the Harish-Chandra map \cite{Ros90,Tan92}, which identifies the centre with the character ring of the Lie algebra $\mathfrak{g}$.  Joseph and Letzter \cite{JL94} further analysed the ad-finite part of $\mathrm{U}_q(\mathfrak{g})$ under the adjoint action, showing that it decomposes into blocks indexed by dominant integral weights, and each block contains a unique line defining a central element.  Using Joseph–Letzter’s description, Baumann \cite{Bau98} proved that certain central elements constructed by Reshetikhin \cite{Res90} generate  the centre in types $\mathsf{A}$  and $\mathsf{C}$, while in types $\mathsf{B}$ and $\mathsf{D}$ they generate a subalgebra of the centre fixed by an involution.

Recently, there has been a growing interest in constructing explicit generators  for the centre of quantum groups, motivated by advances in mathematical physics and probability \cite{CGRS16,Kua18},  but  primarily focused on type $\mathsf{A}$ \cite{Li10, LXZ18, Dai22, KZ23}. The construction of Zhang, Bracken and Gould \cite{ZGB91} gives rise to two kinds of generating sets for the centre $\mathcal{Z}(\mathrm{U}_q(\mathfrak{g}))$, each consisting of  elements $C_{V,\ell}$:
(i) with the  order $\ell$ fixed, and the representation $V$ varying; or (ii) with $V$ fixed, and $\ell$ varying. It is well known \cite{Dri90, Res90} that the centre   is generated by elements $C_{\varpi_i}$ indexed by the fundamental representations $\varpi_i$ of $\mathfrak{g}$. This corresponds to the first kind of generating set, with the order $\ell=1$ and $V$ ranging over all fundamental representations.  

The second kind of generating set is more subtle and intriguing. In type $\mathsf{A}_n$, it was shown in \cite{Li10} that the centre  is generated by higher-order quantum Casimir elements $C_{V,\ell}, 1\leq \ell\leq n,$ associated with the natural representation $V$. These elements are quantum analogues of the classical higher-order Casimir elements $\sum_{i_1=1}^n \cdots \sum_{i_\ell=1}^n E_{i_1i_2}E_{i_2i_3}\cdots E_{i_{\ell-1}i_{\ell}} E_{i_{\ell}i_1}$ with $\ell\geq 1$ in the universal enveloping algebra $\mathrm{U}(\mathfrak{gl}_n)$, where $E_{ij}$ are the matrix units of $\mathfrak{gl}_n$. Remarkably, the eigenvalues of the higher-order quantum Casimir elements on simple modules can be expressed through irreducible $\mathfrak{gl}_n$-characters associated with hook partitions (see \figref{fig: youngdia}).

The purpose of this paper is to construct a generating set for the centre in terms of higher-order quantum Casimir elements for  the remaining classical types $\mathsf{B}$, $\mathsf{C}$ and $\mathsf{D}$, thus presenting a comprehensive picture of the connections between higher-order quantum Casimir elements and irreducible characters associated with hook partitions in all classical types.

\subsection{Main results and approach} 
We are concerned with quantum groups associated with classical Lie algebras  of types $\mathsf{B}_n\, (n\geq 2)$, $\mathsf{C}_n\,(n\geq 3)$ and $\mathsf{D}_n\, (n\geq 4)$. The higher-order quantum Casimir elements in the sense of Zhang, Bracken, and Gould (\secref{sec: highCas}) are of the form $C_{V,\ell}$, where $1\leq \ell\leq n$ and  $V$ is fixed to be the natural representation for each type.

 Our first main result, stated  in \thmref{thm: CentreGen},  is  summarised as follows. For type $\mathsf{C}_n\,(n\geq 3)$,  the centre is generated by higher-order quantum Casimir elements $C_{V,\ell}$ for $1\leq \ell\leq n$.  In type $\mathsf{B}_n\,(n\geq 2)$, the  centre  is generated by $C_{V,\ell}$ for $1\leq \ell\leq n-1$,  together with the order-one quantum Casimir element $C_{S,1}$, where $S$ is the spin representation.  In type $\mathsf{D}_n\,(n\geq 4)$, the  centre  is generated by $C_{V,\ell}$ for $1\leq \ell\leq n-2$, together with the order-one quantum Casimir elements  $C_{S_+,1}$ and  $C_{S_-,1}$, where $S_+$ and $S_{-}$ denote the two half-spin representations.   

To prove this result, we analyse the images of these higher-order quantum Casimir elements under the quantum Harish-Chandra isomorphism. These images are determined by the eigenvalues of the higher-order quantum Casimir elements acting on highest weight modules, which can be derived from results in \cite{DGL05}. We then reformulate these images in terms of irreducible characters of $\mathfrak{g}$ associated with hook partitions (see Theorems \ref{thm: ChGtypeB}, \ref{thm: ChGtypeC} and \ref{thm: ChGtypeD}). This reformulation is inspired by the type $\mathsf{A}$ case \cite{Li10}, but requires new techniques due to the more intricate nature of representation theory in types $\mathsf{B}$, $\mathsf{C}$ and $\mathsf{D}$.  Finally, we apply invariant theory for the Weyl group $W$ of $\mathfrak{g}$, together with the Jacobi-Trudy identities in the theory of symmetric functions, to complete the proof.

A new idea in our approach is the introduction of  an auxiliary expression $H_{n,k}$, which enables a representation-theoretic interpretation of the   higher-order quantum Casimir elements. For any integer $k\geq 0$, this expression is defined uniformly for all classical types by
\[H_{n,k}:= e^{\rho+k\varepsilon_1} \prod_{\alpha\in \Phi_{\varepsilon_1}^+}(1-q^{-2(\alpha, \varepsilon_1)}e^{-\alpha}), 
\]   
where $\rho$ denotes the half sum of positive roots, and the product is taken over all positive roots $\alpha$ satisfying $(\alpha, \varepsilon_1) \neq 0$, where $\varepsilon_1$ is the highest weight of the natural representation. Applying the antisymmetriser of the Weyl group of $\mathfrak{g}$ to $H_{n,k}$ yields, up to a factor of the Weyl denominator, a linear combination of irreducible characters indexed by hook partitions. This reformulation not only provides the desired representation-theoretic interpretation, but also exhibits a representation stability phenomenon  \cite{CF13} (see \rmkref{rem: stab}). This result is of independent interest and may have further applications, for instance in the theory of special functions \cite{Eti99}.

The construction of higher-order quantum Casimir elements extends naturally to quantum supergroups \cite{ZG91}, and their eigenvalues on highest weight  modules have been determined explicitly for $\mathrm{U}_q(\mathfrak{gl}_{m|n})$ and $\mathrm{U}_q(\mathfrak{osp}_{m|2n})$ \cite{LZ93, DGL05}. More recently, the Harish-Chandra isomorphism for quantum supergroups has been established in \cite{LWY22}, providing a precise description for the  algebraic structure of the centre of quantum supergroups.  We expect our representation-theoretic approach, especially the uniform construction of $H_{n,k}$,  applies immediately to  quantum supergroups. This would yield new insights into the connections between higher-order quantum super Casimir elements and irreducible characters of Lie superalgebras associated with partitions of certain types. We leave these for future work.

This paper is organised as follows. In Section \ref{sec: QuanGrp}, we cover the necessary definitions and the construction of  higher order quantum Casimir elements. In Section \ref{sec: eigenvalue}, we  derive their eigenvalues when acting on simple modules and hence determine their images under the Harish-Chandra isomorphism. In Section \ref{sec: repform}, we establish a bridge between the eigenvalues of these quantum  Casimir elements and the characters of finite-dimensional representations associated with hook partitions.   Finally, in Section \ref{sec: centre},  we prove the main theorem, which provides an explicit description of the generating set for the center in terms of these higher-order quantum Casimir elements.

\subsection{Notation} Throughout the paper, we write $\mathbb{Z}$ for the set of integers, $\mathbb{Q}$ for the set of rational numbers, and $\mathbb{C}$ for the set of complex numbers. 
For any real number $x$, we denote by $\lfloor x\rfloor$ the largest integer less than or equal to $x$. For any two integers $i,j$, let $\delta_{i,j}$ denote the Kronecker delta function, which equals 1 if $i=j$ and 0 otherwise.   To distinguish objects of types $\mathsf{B}_n$, $\mathsf{C}_n$ and $\mathsf{D}_n$, we append the corresponding subscript to the symbol. For instance, $\rho_{\mathsf{B}_n}$, $\rho_{\mathsf{C}_n}$ and $\rho_{\mathsf{D}_n}$ denote the half-sum of positive roots in types $\mathsf{B}_n$, $\mathsf{C}_n$, and $\mathsf{D}_n$, respectively.
Unless otherwise stated in the context, a symbol without referring to a specific type will be understood to apply to all three types.

\vspace{0.2cm}
\noindent{\bf Acknowledgment.}  We are grateful to Mark Gould,  Ole Warnaar and Ruibin Zhang for valuable discussions and suggestions. 

\section{Preliminaries}\label{sec: QuanGrp} 

\subsection{The Drinfeld-Jimbo quantum group} We refer to \cite{Jan96} for basics on the Drinfeld-Jimbo quantum group. Let $\mathfrak{g}$ be a finite-dimensional simple Lie algebra of rank $n$ over $\mathbb{C}$, $\mathfrak{h}$ a fixed Cartan subalgebra, and $\Phi$ the associated root system. Fix a positive root system $\Phi^+$ and simple roots $\Pi=\{\alpha_1, \dots, \alpha_n\}$.  Let $\{h_1, \dots, h_n\}$ be a basis of $\mathfrak{h}$ and $\{\varepsilon_1, \dots, \varepsilon_n\}$ its dual basis in $\mathfrak{h}^*$ satisfying $\varepsilon_i(h_j)=\delta_{ij}$ for $1\leq i, j\leq n$. 

Define a non-degenerate symmetric bilinear form $(\cdot, \cdot)$ on $\mathfrak{h}^*$ by $(\varepsilon_i, \varepsilon_j)=\delta_{ij}$ for $1\leq i,j\leq n$. 
For types $\mathsf{B}_n\,(n\geq 2), \mathsf{C}_n\,(n\geq 3),$ and $\mathsf{D}_n\,(n\geq 4)$, the positive root system $\Phi^+$,  simple root system $\Pi=\{\alpha_1, \dots, \alpha_n\}$, and the half sum $\rho$ of positive roots are given as follows.
\begin{align*}
\text{Type } \mathsf{B}_n:\quad 
& \Phi^+ = \{ \varepsilon_i \pm \varepsilon_j \mid 1 \leq i < j \leq n \} \cup \{ \varepsilon_i \mid 1 \leq i \leq n \}, \\
& \Pi = \{ \varepsilon_1 - \varepsilon_2,\ \varepsilon_2 - \varepsilon_3,\ \dots,\ \varepsilon_{n-1} - \varepsilon_n,\ \varepsilon_n \}, \quad \rho= \sum_{i=1}^n(n-i+\frac{1}{2})\varepsilon_i;\\ 
\text{Type } \mathsf{C}_n:\quad 
& \Phi^+ = \{ \varepsilon_i \pm \varepsilon_j \mid 1 \leq i < j \leq n \} \cup \{ 2\varepsilon_i \mid 1 \leq i \leq n \}, \\
& \Pi = \{ \varepsilon_1 - \varepsilon_2,\ \dots,\ \varepsilon_{n-1} - \varepsilon_n,\ 2\varepsilon_n \}, \quad\quad \quad \quad   \rho=\sum_{i=1}^n (n-i+1)\,\varepsilon_i;\\
\text{Type } \mathsf{D}_n:\quad 
& \Phi^+ = \{ \varepsilon_i \pm \varepsilon_j \mid 1 \leq i < j \leq n \}, \\
& \Pi = \{ \varepsilon_1 - \varepsilon_2,\ \dots,\ \varepsilon_{n-1} - \varepsilon_n,\ \varepsilon_{n-1} + \varepsilon_n \}, \quad \ \rho=\sum_{i=1}^n (n-i)\,\varepsilon_i.
\end{align*}
Let $W$ denote the Weyl group of $\Phi$ generated by simple reflections $s_{\alpha_i}, 1\leq i\leq n,$ defined by 
\[s_{\alpha_i}(v)= v- 2 \frac{(v,\alpha_i)}{(\alpha_i,\alpha_i)}\alpha_i, \quad v\in \mathfrak{h}^*.  \]

For $1\leq i\leq n$, denote by $\varpi_i$ the fundamental weight such that $2(\varpi_i, \alpha_j)/ (\alpha_j, \alpha_j)=\delta_{ij}$ for all $1\leq j\leq n$. Let  $P=\bigoplus_{i=1}^{n} \mathbb{Z}\varpi_i$ be the weight lattice  and $P^{+}=\bigoplus_{i=1}^{n} \mathbb{Z}_{\geq 0}\varpi_i$ be the set of dominant integral weights of $\mathfrak{g}$. Let $m_{o}$ be the minimal positive integer such that $(\lambda, \mu)m_{o} \in\mathbb{Z}$ for all $\lambda, \mu\in P$. Let $\hat{q}=q^{\frac{1}{m_o}}$ be an indeterminate, and let $\mathbb{K}=\mathbb{C}(\hat{q})$  be the field of rational functions  in $\hat{q}$. 
For any $\lambda, \mu\in P$, the expression $q^{(\lambda, \mu)}$ is understood as $\hat{q}^{(\lambda, \mu)m_o}$.

The Drinfled-Jimbo quantum group ${\rm U}_q(\fg)$  is a unital associative algebra over $\mathbb{K}$ generated by $E_i,F_i,K_{\lambda}$ for $1\leq i\leq n$ and  $\lambda\in P$, subject to the relations:  
\begin{align*}
&K_0=1, \quad K_{\lambda}K_{\mu}= K_{\lambda+\mu},  \\ 
&K_{\lambda}E_iK_{-\lambda}= q^{(\lambda, \alpha_i)}E_i,\\ 
&K_{\lambda}F_iK_{-\lambda}= q^{-(\lambda, \alpha_i)}F_i,\\ 
&E_iF_i- F_iE_i= \delta_{ij} \frac{K_{\alpha_i}-K_{-\alpha_i}}{q_i-q_i^{-1}}, \\
&  \sum_{k=0}^{1-a_{ij}}(-1)^k
\begin{bmatrix}1-a_{ij}\\ k\end{bmatrix}_{q_i}
E_i^{\,1-a_{ij}-k}E_jE_i^{\,k}=0\quad(i\ne j),\\ 
&  \sum_{k=0}^{1-a_{ij}}(-1)^k
\begin{bmatrix}1-a_{ij}\\ k\end{bmatrix}_{q_i}
F_i^{\,1-a_{ij}-k}F_jF_i^{\,k}=0\quad(i\ne j),
\end{align*}
where $q_i=q^{(\alpha_i, \alpha_i)/2}$ for $1\leq i\leq n$, $a_{ij}= 2(\alpha_i, \alpha_j)/(\alpha_i, \alpha_i)$ for $1\leq i,j \leq n$, and 
\[
[m]_{q_i}=\dfrac{q_i^{\,m}-q_i^{-m}}{q_i-q_i^{-1}},\quad [m]_{q_i}!=[m]_{q_i}\cdots [1]_{q_i},\quad 
\begin{bmatrix}n\\ k\end{bmatrix}_{q_i}
=\frac{[n]_{q_i}!}{[k]_{q_i}!\,[n-k]_{q_i}!}    \]
 for any integers $m\geq 0$ and $n\geq k\geq 0$. It is well-known that ${\rm U}_q(\fg)$ is a Hopf algebra, with the co-multiplication $\Delta$, co-unit $\epsilon$, and antipode $S$ defined by 
\begin{eqnarray*}
&\Delta(K_{\lambda})=K_{\lambda}{\otimes}K_{\lambda}, \quad
\Delta(E_{i})=K_{\alpha_i}{\otimes}E_{i}{+}E_{i}{\otimes}1, \quad
\Delta(F_{i})=F_{i}{\otimes}K_{\alpha_i}^{-1}{+}1{\otimes}F_i,\\
&\epsilon (K_{\lambda})=1,\quad \epsilon (E_{i})=0,
\quad \epsilon({F_{i}})=0,\\
&S(K_{\lambda})=K_{\lambda}^{-1},\quad S(E_{i})={-}K_{\alpha_i}^{{-}1}E_{i},\quad S(F_{i})={-}F_{i}K_{\alpha_i}.
\end{eqnarray*}
for $\lambda \in P$ and $1\leq i\leq n$.

For convenience, we write $\U=\U_q(\fg)$. We denote by $\U^+$ (resp. $\U^-$) the subalgebra of $\U$ generated by all $E_i$ (resp. $F_i$),  and by $\U^0$ the subalgebra generated   by $K_{\lambda}$ for all $\lambda\in P$. By the PBW theorem, multiplication induces a vector space  isomorphism $\mathrm{U}^- \otimes \U^0 \otimes \U^+\cong \U$.

Throughout the paper, we are only concerned with type 1 finite-dimensional $\U$-modules, that is, modules $M$  with weight space decomposition $M= \oplus_{\lambda\in P} M_{\lambda}$, where $M_{\lambda}=\{v\in M \mid K_{\mu}v= q^{(\mu,\lambda)}v, \forall \mu\in P\}$. It is well known that every finite-dimensional $\U$-module is completely reducible,  and each finite-dimensional simple $\U$-module is a highest weight module with highest weight in $P^+$ \cite{Jan96}.

\subsection{The Harish-Chandra isomorphism}
 We shall describe the centre $\mathcal{Z}({\rm U})$ of ${\rm U}$ in terms of the quantum analogue of the  Harish-Chandra isomorphism \cite{Ros90, Tan92}.

Let $Q= \bigoplus_{i=1}^n \mathbb{Z}\alpha_i$ be the root lattice of $\mathfrak{g}$. Then, $\U=\bigoplus_{\nu\in Q}\U_{\nu}$ is a $Q$-graded algebra, where
\[
    \U_{\nu} =\{u\in \U\mid K_{\lambda}u K_{-\lambda}= q^{(\lambda,\nu)}u, \  \forall \lambda\in P \}. 
\]
In particular, we have $\U_0=\U^0 \oplus \bigoplus_{\nu>0} \U_{-\nu}^- \U^0 \U_{\nu}^+$, where  $\U_{\nu}^+= \U^+\cap \U_{\nu}$ and $\U_{-\nu}^-= \U^-\cap \U_{-\nu}$ for any positive $\nu\in Q$. Clearly, we have $\mathcal{Z}(\U)\subseteq \U_0$, and the linear  projection 
\[\pi: \U_0\longrightarrow \U^0 \]
is an algebra homomorphism.

Let ${\rm U}_{ev}^0$ be the subalgebra of ${\rm U}^0$ generated by  $K_{2\lambda}$ for all $\lambda\in P$. The Weyl group $W$ acts on  ${\rm U}^0$ by  $wK_{\lambda}= K_{w(\lambda)}$ for all $\lambda \in P$. We denote by  
\[
({\rm U}_{ev}^0)^{W}:=\{ h\in {\rm U}_{ev}^0 \mid wh=h,\, \forall w\in W \}. 
\]
the subalgebra of $W$-invariants in ${\rm U}_{ev}^0$. Define an algebra automorphism
\begin{equation*}
\gamma_{-\rho}:{\rm U}^0\to{\rm U}^0,\quad  K_{\mu}\mapsto q^{-(\rho,\mu)}K_{\mu}, \quad \mu \in P. 
\end{equation*}
Then the Harish-Chandra homomorphism $\gamma_{-\rho}\circ \pi: \U_0\rightarrow \U^0$ restricts to the centre $\mathcal{Z}(\U)\subseteq \U_0$. 

\begin{theorem}\label{thm:HCiso}\cite{Ros90,Tan92}
The  Harish-Chandra   homomorphism
 \begin{equation*}\label{eq: HCmap}
  \varphi:=\gamma_{-\rho}\circ\pi:\mathcal{Z}({\rm U})\rightarrow({\rm U}_{ev}^0)^W
 \end{equation*}
is an algebra isomorphism.
\end{theorem}

\subsection{Higher-order quantum Casimir elements}\label{sec: highCas} 
We now introduce higher-order quantum Casimir elements of $\mathrm{U}_q(\mathfrak{g})$ following  \cite{ZGB91}. This makes use of the quasi-$R$-matrix  $\mathfrak{R}$ of $\U_q(\fg)$, which is an infinite sum of the form 
\[
 \mathfrak{R}= \sum_{i} F^{(i)}\otimes E^{(i)}, 
 \]
 where $F^{(i)}\in \U^-$ and $E^{(i)}\in \U^+$ for all $i$. We refer to \cite[Section 8.3.3]{KS12} for the explicit formula.

Let $\zeta: \U_q(\mathfrak{g})\rightarrow \mathrm{End}_{\mathbb{K}}(V)$ be a finite dimensional representation.
Define
\begin{align*}
\mathcal{R}_V:=(\zeta\otimes {\rm id})(\mathfrak{R}), \quad
\mathcal{R}_{V}^T:=(\zeta\otimes {\rm id})\mathfrak{R}^T,
\end{align*}
where $\mathfrak{R}^T:= \sum_{i}E^{(i)}\otimes F^{(i)}$. 
Then,  both $\mathcal{R}_{V}$ and $\mathcal{R}_{V}^T$ are finite sums in $\End(V)\otimes \U_q(\fg)$. 
Let $V=\bigoplus_{\mu\in \Pi(V)} V_{\mu}$ be the weight space decomposition, where $\Pi(V)$ is the set of all weights $\mu\in P$ such that $V_{\mu}\neq 0$.   For any weight $\mu$ of $V$, let $\mathsf{P}_{\mu}:V\to V_{\mu}$ be the linear projection. We  define 
\begin{align*}
&\mathcal{K}_{V}:=\sum_{\mu\in P}\mathsf{P}_{\mu}\otimes K_{\mu}\in \End(V)\otimes\U_q(\fg), \\ 
&L_V:=\mathcal{K}_{V}\mathcal{R}_{V}, \quad \quad L_V^T:=\mathcal{K}_{V}\mathcal{R}_{V}^T.
\end{align*}

\begin{lemma}\cite[Theorem 3.5]{Dai22}\label{lem: Gam}
For any finite dimensional representation $\zeta: \U_q(\mathfrak{g})\rightarrow \mathrm{End}_{\mathbb{K}}(V)$, define
\[\Gamma_{V}:=L_V^T L_V.\]
Then $\Gamma_{V}\in {\rm End}(V)\otimes \U_q(\fg)$ commutes with $(\zeta\otimes {\rm id})\Delta(x)$ for all $x\in\U_q(\fg)$.
\end{lemma}

We define the partial trace
\[
\begin{aligned}
&\textnormal{Tr}_1: \End(V)\otimes {\rm U}_q(\fg) \longrightarrow \U_q(\fg),\quad  f\otimes x\mapsto \text{Tr}(f)x
\end{aligned}
\]
for any $f\in\End(V), x\in {\rm U}_q(\fg)$, where ${\rm Tr}$ is the usual trace operator on $\End(V)$. For any integer  $\ell\geq 0$, we define the  \emph{quantum Casimir element} of order $\ell$ associated with $V$ to be 
\begin{equation*}
C_{V,\ell}={\rm Tr}_1\left((\zeta(K_{2\rho})\otimes 1)\left(\frac{\Gamma_{V}-1\otimes 1}{q-q^{-1}}\right)^{\ell}\right). 
\end{equation*}

\begin{proposition}
The quantum Casimir elements $C_{V,\ell}$ are central in $\U_q(\fg)$ for all $\ell\geq 0$. 
\end{proposition}

\begin{proof}
By \lemref{lem: Gam}, $\Gamma_V$ commutes with $(\zeta\otimes {\rm id})\Delta(x)$ for all $x\in\U_q(\fg)$, so does the operator $\left((\Gamma_V-1\otimes 1)/(q-q^{-1})\right)^\ell$ for any integer $\ell\geq 0$.  It follows from \cite[Proposition 1]{ZGB91} that $C_{V,\ell}$ is central. 
\end{proof}

\begin{remark}\label{rmk: qdim}
 If $\ell=0$, then $C_{V,0}= \mathrm{Tr}(\zeta(K_{2\rho}))= \mathrm{qdim}(V)$, the quantum dimension of $V$.  
\end{remark}

A natural question is whether these quantum Casimir elements generate the centre $\mathcal{Z}(\U)$. Indeed,  $\mathcal{Z}(\U)$ is a polynomial algebra generated by  the algebraically independent elements
\[ C_{\varpi_i}:=C_{L(\varpi_i), 1}, \quad 1\leq i\leq n,  \]
 where $L(\varpi_i)$ denotes the $i$-th fundamental representation of $\U$; see, e.g., \cite{Dai22}. More generally, for any $\lambda\in P^+$, denote $C_{\lambda}:= C_{L(\lambda),1}$  and $C_{\lambda}^0:= \varphi(C_{\lambda})$. Then we have 
\[
C_{\lambda}^0=\sum_{\mu\in \Pi(L(\lambda))} m_{\lambda}(\mu)K_{2\mu}, 
\]
where $m_{\lambda}(\mu)= \mathrm{dim}\,  L(\lambda)_{\mu}$ (see, e.g., \cite[Proposition 3.8]{Dai22}). 

\begin{theorem}\label{thm: generators of Uev0}
 The the $W$-invariant subalgebra $(\U_{ev}^0)^W$ of $\U_{ev}^0$ is a polynomial algebra over $\mathbb{K}$   generated by   $C^0_{\varpi_1}, C^0_{\varpi_2},\dots, C^0_{\varpi_n}$.
\end{theorem}

Another possible generating set of  $\mathcal{Z}(\U)$ contains  higher-order quantum Casimir elements, obtained  by fixing a module $V$ and varying $\ell$. In the remainder of this paper, we focus on  the following higher-order quantum Casimir  elements 
\[C_{n, \ell}:=C_{L(\varpi_1), \ell}, \quad \ell \geq 1, \]
 where $L(\varpi_1)$ is the natural module of $\U$.

\section{Eigenvalues of higher-order quantum  Casimir elements}\label{sec: eigenvalue}

Throughout this section, for any integer $n\geq 1$,  we define
\[\mathbf{I}_n:=\{-n, \dots, -1,1, \dots, n\}. \]
Let $L(\lambda)$ be a finite dimensional simple $\U_q(\mathfrak{g})$-module with highest weight $\lambda\in P^+$.   Since $C_{n,\ell}$ is central for any integer $\ell\geq 1$, it acts on $L(\lambda)$ as multiplication by  a scalar, called the \emph{eigenvalue} of $C_{n,\ell}$ on $L(\lambda)$ and denoted by $\omega_{\lambda}(C_{n,\ell})$.

\begin{proposition}\label{prop: eigval} \cite[Theorem 4.1]{DGL05}
For any $\lambda\in P^+$ and integer $\ell\geq 1$, the eigenvalues $\omega_{\lambda}(C_{n,\ell})$ for quantum groups of types $\mathsf{B}_n\,(n\geq 2), \mathsf{C}_n\, (n\ge 3), \mathsf{D}_n\, (n\ge 4)$ are given by
    \begin{equation*}
        \omega_{\lambda}(C_{n,\ell})=\sum_{a\in \mathbf{I}'_n}q^{c_n-(\varepsilon_a,\varepsilon_a)}f(a)\left(\frac{q^{(\varepsilon_a,2\rho+2\lambda+\varepsilon_a)-c_n}-1}{q-q^{-1}}\right)^{\ell}
            \prod_{b\in \mathbf{I}'_{n}\setminus\{a\}}\frac{q^{(\varepsilon_a,2\rho+2\lambda+\varepsilon_a)}-q^{(\varepsilon_b,2\rho+2\lambda-\varepsilon_b)}}{q^{(\varepsilon_a,2\rho+2\lambda+\varepsilon_a)}-q^{(\varepsilon_b,2\rho+2\lambda+\varepsilon_b)}}, 
    \end{equation*}
where $\varepsilon_{-a}:=-\varepsilon_a$ for $a>0$, $\varepsilon_0:=0$ in type $\mathsf{B}_n$, and  $\mathbf{I}'_{n}, c_n$ and $f(a)$ are defined as follows:
 \begin{align*}
 \text{Type }\mathsf{B}:\quad
&\mathbf{I}_n' = \mathbf{I}_n \cup \{0\}, 
&& c_n = 2n, 
&& f(a) = 1 + (1 - \delta_{a,0})(q - q^{-1})
   \frac{q^{(\varepsilon_a,\varepsilon_a+2\rho+2\lambda)}}
        {q^{(2\varepsilon_a,\varepsilon_a+2\rho+2\lambda)} - 1},
\\
\text{Type }\mathsf{C}:\quad
&\mathbf{I}_n' = \mathbf{I}_n,            
&& c_n = 2n + 1, 
&& f(a) = 1 + \frac{1 - q^{-2}}
               {q^{(2\varepsilon_a,\varepsilon_a+2\rho+2\lambda)} - 1},
\\
\text{Type }\mathsf{D}:\quad
&\mathbf{I}_n' = \mathbf{I}_n,            
&& c_n = 2n - 1, 
&& f(a) = 1 - \frac{q^2 - 1}
               {q^{(2\varepsilon_a,\varepsilon_a+2\rho+2\lambda)} - 1}.
 \end{align*}
\end{proposition}
\begin{proof}
The  construction of  higher-order Casimir elements  elements extends to the quantum supergroup $\mathrm{U}_q(\mathfrak{osp}_{m|2n})$ \cite{DGL05}, which specialises to $\mathrm{U}_q(\mathfrak{so}_m)$ when $n=0$ and to $\mathrm{U}_q(\mathfrak{sp}_{2n})$ when $m=0$.
The  formula in  \cite[Theorem 4.1]{DGL05} gives  explicit expressions for the eigenvalues of higher-order  Casimir elements in $\mathrm{U}_q(\mathfrak{osp}_{m|2n})$ for $m>2$. In particular, setting $n=0$ yields the eigenvalues $\omega_{\lambda}(C_{\lfloor m/2\rfloor,\ell})$ for $\mathrm{U}_q(\mathfrak{so}_m)$ with $m>2$. For $\mathrm{U}_q(\mathfrak{sp}_{2n})$, the method of \cite{DGL05} applies without essential change. We do not reproduce the lengthy computations here, but only state the resulting formula. 
\end{proof}

\begin{remark}
In \cite[Proposition 4]{GZB91}, the central element $C_{V, \ell}$ is constructed in a slightly different way  for any finite-dimensional $\mathrm{U}_q(\mathfrak{g})$-module $V$ and integer $\ell\geq 1$, with eigenvalues given explicitly in terms of quantum dimensions. In this paper, we do not compare the two versions of $C_{V, \ell}$ or their eigenvalues, and instead adopt the eigenvalue formula in \propref{prop: eigval} for all subsequent discussions, as this  extends more naturally to quantum supergroups, which we aim to pursue in future work. 
\end{remark}

For any integer $\ell\geq 1$, we write the image of $C_{n,\ell}$ under the Harish-Chandra isomorphism as 
 \[C^0_{n,\ell}:=\varphi(C_{n,\ell})\in (\U_{ev}^0)^W. \]
We define 
\[L_a:=q^{(\varepsilon_a,2\rho)}K_{2\varepsilon_a}, \quad \text{for each } a\in \mathbf{I}'_n.\]
Note also that $L_{-a}= L^{-1}_a$ for any $a\in \mathbf{I}'_n$. 
For any weight vector  $v_{\mu}$ with  $\mu\in P$, we have $L_{a}v_{\mu}= q^{(2\varepsilon_a,\, \mu+\rho)}v_{\mu}$.

\begin{lemma}\label{lem: C^0_nl}
For any integer $\ell\geq 1$, we have the following formulas:
\begin{align*}
\text{Type }\mathsf{B}_n (n\geq 2):
C^0_{n,\ell}
&=
\sum_{a\in\mathbf I_n}
  \frac{qL_a - q^{-1}L_a^{-1} + q - q^{-1}}{L_a - L_a^{-1}}
  \left(\frac{q^{1-2n}L_a - 1}{q - q^{-1}}\right)^{\!\ell}P_{n,a}+ \left(\frac{q^{-2n} - 1}{q - q^{-1}}\right)^{\!\ell},
\\
\text{Type }\mathsf{C}_n (n\geq 3):
C^0_{n,\ell}
&=
\sum_{a\in \mathbf{I}_n}
\frac{q^2L_a - q^{-2}L_a^{-1}}{L_a - L_a^{-1}}
\left(\frac{q^{-2n}L_a - 1}{q - q^{-1}}\right)^{\!\ell}P_{n,a},
\\
\text{Type }\mathsf{D}_n (n\geq 4):
C^0_{n,\ell}
&=
\sum_{a\in \mathbf{I}_n}
\left(\frac{q^{2-2n}L_a - 1}{q - q^{-1}}\right)^{\!\ell}P_{n,a},
\end{align*}
where 
\[
   P_{n,a}= \prod_{b\in\mathbf I_n\setminus\{\pm a\}}
    \frac{qL_a - q^{-1}L_b}{L_a - L_b}, \quad a\in \mathbf{I}_n. 
\] 
\end{lemma}

\begin{proof}
Recall that $\mathcal{Z}(\mathrm{U})\subseteq \mathrm{U}_0=\U^0 \oplus \bigoplus_{\nu>0} \U_{-\nu}^- \U^0 \U_{\nu}^+$. When acting on a highest weight vector in a simple $\mathrm{U}$-module, only the $\U^0$-component of $C_{n,\ell}$ contributes to the eigenvalue of $C_{n,\ell}$. Therefore, by \propref{prop: eigval}, we have 
\[
    C^0_{n,\ell}= \sum_{a\in \mathbf{I}'_{n}}q^{c_n-(\varepsilon_a, \varepsilon_a)}F(L_a) \left(\frac{q^{(\varepsilon_a,\varepsilon_a)-c_n}L_a-1}{q-q^{-1}}\right)^{\ell}\prod_{b\in \mathbf{I}'_{n}\setminus\{a\}} \frac{q^{(\varepsilon_a,\varepsilon_a)}L_a-q^{-(\varepsilon_b,\varepsilon_b)}L_b}{q^{(\varepsilon_a,\varepsilon_a)}L_a-q^{(\varepsilon_b,\varepsilon_b)}L_b},
\]
where $c_n$ and $\mathbf{I}'_n$ are defined as in \propref{prop: eigval}, and 
\[
F(L_a)=
\left\{
\begin{array}{l@{\qquad}l}
\displaystyle 1+(1-\delta_{a,0})(q-q^{-1})
 \frac{qL_a}{q^2L_a^2-1},
 & \text{type } \mathsf{B}_n,\\[1ex]
\displaystyle 1+\frac{1-q^{-2}}
      {q^2L_a^2-1},
 & \text{type } \mathsf{C}_n,\\[1ex]
\displaystyle 1-\frac{q^{2}-1}
      {q^2L_a^2-1},
 & \text{type } \mathsf{D}_n.
\end{array}
\right.
\]
In the case of type $\mathsf{B}_n$, splitting the sum into  the cases $a\in \mathbf{I}_n$ and $a=0$ yields
\begin{align*}
 C^0_{n,\ell}&= \sum_{a\in \mathbf{I}_n} \left( 1+ \frac{q-q^{-1}}{qL_a-q^{-1}L_{a}^{-1}} \right) \left(\frac{q^{1-2n}L_a-1}{q-q^{-1}}\right)^{\ell} \prod_{b\in \mathbf{I}_n\setminus\{a\}}  \frac{qL_a-q^{-1}L_b}{L_a-L_b}\\ 
 & \quad \ + q^{2n} \left(\frac{q^{-2n}-1}{q-q^{-1}}\right)^{\ell} \prod_{b\in \mathbf{I}_n}\frac{1-q^{-1}L_{b}}{1-qL_b} \\
 &=\! \sum_{a\in \mathbf{I}_n}\! \left( \frac{qL_a-q^{-1}L_a^{-1}+ q-q^{-1}}{L_a-L_a^{-1}} \right) \!\left(\frac{q^{1-2n}L_a-1}{q-q^{-1}}\right)^{\ell} P_{n,a} + \left(\frac{q^{-2n}-1}{q-q^{-1}}\right)^{\ell}, 
 %&\quad \ +\left(\frac{q^{-2n}L_a-1}{q-q^{-1}}\right)^{\ell}. 
\end{align*}
where we have used the following identities 
\begin{align*}
\prod_{b\in \mathbf{I}'_n\setminus \{a\}}\frac{q^{(\varepsilon_a,\varepsilon_a)}L_a-q^{-(\varepsilon_b,\varepsilon_b)}L_b}{q^{(\varepsilon_a,\varepsilon_a)}L_a-q^{(\varepsilon_b,\varepsilon_b)}L_b} &=q^{1-2n} \prod_{b\in \mathbf{I}_n\setminus\{a\}}  \frac{qL_a-q^{-1}L_b}{L_a-L_b}, \quad a\in \mathbf{I}_n, \\ 
q^{2n}\prod_{b\in \mathbf{I}_n}\frac{1-q^{-1}L_b}{1-qL_b}&=1. 
\end{align*}
The formulas for the other two cases can be obtained similarly. 
\end{proof}

Applying the binomial expansion to \lemref{lem: C^0_nl}, we  obtain the following.  

\begin{corollary}\label{coro: C^0_nl}
For any integer $\ell \geq 1$, the element $C^0_{n,\ell}$ can be written as
\[
C^0_{n,\ell}= \frac{1}{(q^{-1}-q)^{\ell}} \sum_{k=0}^{\ell} \binom{\ell}{k} (-q^{1-c_n})^{k}G_{n,k},
 \]
where $c_n$ is as defined in \propref{prop: eigval}, and for $0\leq k\leq \ell$, the terms $G_{n,k}$ are given by
\begin{align*}
\text{Type }\mathsf{B}_n\, (n\geq 2):\quad
G_{n,k}
&=  q^{-k}+ 
\sum_{a\in \mathbf{I}_n} \frac{qL_a-q^{-1}L_a^{-1}+q-q^{-1}}{L_a-L_a^{-1}} L_a^k P_{n,a}, \\
\text{Type }\mathsf{C}_n\, (n\geq 3):\quad
G_{n,k}
&=
\sum_{a\in\mathbf I_n}
\frac{q^2L_a - q^{-2}L_a^{-1}}{L_a - L_a^{-1}}L_a^k  P_{n,a},
\\
\text{Type }\mathsf{D}_n\,(n\geq 4):\quad
G_{n,k}
&=
\sum_{a\in\mathbf I_n}L_a^k P_{n,a}.
\end{align*}
\end{corollary}

It is not obvious that $G_{n,k}$ is a polynomial in $L_{a}, a\in \mathbf{I}_n$ with coefficients in $\mathbb{Z}[q,q^{-1}]$. The equality from \corref{coro: C^0_nl} holds for $\ell\geq 1$, and one would naturally expect it to hold also for $\ell=0$, so that $C^0_{n,0}= G_{n,0}$. On the other hand, by \rmkref{rmk: qdim},  $C_{L(\varpi_1),0}= \mathrm{qdim}(L(\varpi_1))$ is the quantum dimension of the natural module $L(\varpi_1)$ of $\mathrm{U}_q(\mathfrak{g})$. Hence $C^0_{n, 0}= \varphi(C_{L(\varpi_1),0})= \mathrm{qdim}(L(\varpi_1))$. Therefore,  one expects that $G_{n,0}= \mathrm{qdim}(L(\varpi_1))$, which is given  as follows.

\begin{proposition}\label{prop: Gin}
    For types $\mathsf{B}_n\, (n\geq 2)$, $\mathsf{C}_n\, (n\geq 3)$, and $\mathsf{D}_n\, (n\geq 4)$,  we have
\begin{align*}
G_{n,0}=\sum_{a\in \mathbf{I}'_n} q^{(2\rho, \varepsilon_a)}, \quad G_{n,1}= q^{c_n-1}\sum_{a\in \mathbf{I}'_n} L_{a},
\end{align*}
where $c_n$ and $\mathbf{I}'_n$ are defined as in \propref{prop: eigval}.
\end{proposition}
\begin{proof}
The proof is not straightforward, and we do not have a direct proof. Actually, it follows from Theorems \ref{thm: ChGtypeB}, \ref{thm: ChGtypeC} and \ref{thm: ChGtypeD} in the next section. 
\end{proof}

We also note that  $G_{n,1}$ is, up to a  multiple  $q^{c_n-1}$ and a change of variables,  essentially the character of the natural  module. 
For any integer $m>0,$ the transition matrix between the families $\{C^0_{n,\ell}\mid 0\leq \ell\leq m\}$ and $\{G_{n,k}\mid 0\leq k\leq m\}$ turns out to be triangular and invertible over $\mathbb{K}$. It follows that  the images $C_{n,\ell}^0$ of the higher-order quantum Casimir elements $C_{n,\ell}$ are $\mathbb{K}$-linearly spanned by  the $W$-invariant polynomials $G_{n,k}\in (\mathrm{U}^0_{ev})^W$.  In what follows, we will focus primarily  on the elements $G_{n,k}$, and give a representation-theoretic interpretation of these polynomials.

\section{Representation-theoretic reformulation} \label{sec: repform}
In this section, we will reformulate $G_{n,k}$ in terms of characters of finite-dimensional representations of $\fg$. To this end, we need the following expression. 

\begin{notation}
We denote by ${\bf Ch}\, G_{n,k}$ the expression obtained from $G_{n,k}$ by replacing $L_a$ with the formal exponential $e^{\varepsilon_a}$ for all $a \in \mathbf{I}_n=\{-n, \dots, -1,1, \dots, n\}$.
\end{notation}

\subsection{Weyl groups and characters} Recall that the Weyl groups $W_{\mathsf{B}_n}$ and $W_{\mathsf{C}_n}$ of types $\mathsf{B}_n$ and $\mathsf{C}_n$ coincide, and can be identified with the group of all signed permutations on $\mathbf{I}_n=\{-n, \dots, -1,1, \dots, n\}$, that is, bijections $w: \mathbf{I}_n \to  \mathbf{I}_n$ such that $w(-a) = -w(a)$ for $a\in \mathbf{I}_n$. The Weyl group $W_{\mathsf{D}_n}$ of type $\mathsf{D}_n$ is the index-two subgroup consisting of signed permutations with an even number of sign changes. We identify $W_{D_{n-1}}$ (resp. $W_{B_{n-1}}$) as a subgroup of $W_{D_{n}}$ (resp. $W_{\mathsf{B}_n}$) that fixes $\pm 1$. For each $n\geq 2$, we choose the left coset representatives for  $W_{D_{n}}/W_{D_{n-1}}$ as follows: 
\[
R_n= \{ (1,i)(-1,-i), (1, -i)(-1, i) \mid 2\leq i \leq n\}\cup \{(1), (1,-1)(n,-n)\},
\]
where $(1)$ denotes the identity. Indeed, $R_n$ is also a set of left coset representatives for  $W_{B_{n}}/W_{B_{n-1}}$.

Let $\bf{sgn}$ denote the sign character of the Weyl group $W$, and set the antisymmetriser 
\[{\bf A}=\sum_{w\in W}{\bf sgn}(w)w.\] 

\begin{lemma}\label{lem: altvanish}
If   $\lambda\in P$ and $\alpha\in \Phi$ satisfy  $(\lambda, \alpha)=0$, then ${\bf A}(e^{\lambda})=0$.  
\end{lemma}
\begin{proof}
  Let $s_{\alpha}\in W$ be the reflection in the root $\alpha$. Then $s_{\alpha}(\lambda)=\lambda- 2\frac{(\lambda, \alpha)}{(\alpha, \alpha)}= \lambda$. Noting that $\mathbf{sgn}(ws_{\alpha})=-\mathbf{sgn}(w)$, we have  
  \[ \mathbf{A}(e^{\lambda})=  \mathbf{A}(e^{s_{\alpha}(\lambda)})= {\bf A}(s_{\alpha}(e^{\lambda})) = -{\bf A}(e^{\lambda}),   \]
  whence ${\bf A}(e^{\lambda})=0$.
\end{proof}

Let $L(\lambda)$ be a simple $\mathfrak{g}$-module with highest weight $\lambda\in P^+$. By the Weyl character formula, its character is given by 
\[
    \chi(\lambda): = \mathbf{Ch}(L(\lambda))= \frac{\mathbf{A}(e^{\lambda+\rho})}{\mathbf{A}(e^{\rho})}, 
\]
which is $W$-invariant. The  denominator in this expression, denoted by $\Delta$,  is also given by the Weyl denominator  formula
\[
    \Delta=\mathbf{A}(e^{\rho}) = \prod_{\alpha\in \Phi^+} (e^{\frac{\alpha}{2}}-e^{-\frac{\alpha}{2}}).
\]
The following simple observation will be useful later.

\begin{lemma}\label{lem: altiden}
 For any $\lambda\in P^+$, we have $\mathbf{A}(e^{\rho}\chi(\lambda))= \mathbf{A}(e^{\lambda+\rho})$. 
\end{lemma}
\begin{proof}
Since $\chi(\lambda)$ is $W$-invariant, we have $\mathbf{A}(e^{\rho}\chi(\lambda))= \mathbf{A}(e^{\rho})\chi(\lambda) = \mathbf{A}(e^{\lambda+\rho}).$
\end{proof}

\subsection{Reformulation of \texorpdfstring{$G_{n,k}$}{Gnk}}

For integers $n\geq 2$ and $k\geq 0$, we define an auxiliary expression 
\[H_{n,k}:= e^{\rho+k\varepsilon_1} \prod_{\alpha\in \Phi_{\varepsilon_1}^+}(1-q^{-2(\alpha, \varepsilon_1)}e^{-\alpha}), 
\]
where $\Phi_{\varepsilon_1}^+:= \{ \alpha\in \Phi^+\mid (\alpha, \varepsilon_1)>0 \}$. This is defined uniformly for all classical types. In particular,  for  types $\mathsf{B}_n\,(n\geq 2)$, $\mathsf{C}_n\,(n\geq 3)$ and $\mathsf{D}_n\,(n\geq 4)$,  we have
\begin{align*}
H_{\mathsf{B}_n,k} &= e^{\rho_{\mathsf{B}_n}+ k\varepsilon_1}(1-q^{-2}e^{-\varepsilon_1}) \prod_{i=2}^{n}(1-q^{-2}e^{-(\varepsilon_1-\varepsilon_i)})(1-q^{-2}e^{-(\varepsilon_1+\varepsilon_i)}), \\
H_{\mathsf{C}_n,k} &= e^{\rho_{\mathsf{C}_n}+k\varepsilon_1}(1-q^{-4}e^{-2\varepsilon_1}) \prod_{i=2}^{n}(1-q^{-2}e^{-(\varepsilon_1-\varepsilon_i)})(1-q^{-2}e^{-(\varepsilon_1+\varepsilon_i)}), \\
H_{\mathsf{D}_n,k} &= e^{\rho_{\mathsf{D}_n} + k\varepsilon_1} \prod_{i=2}^{n}(1-q^{-2}e^{-(\varepsilon_1-\varepsilon_i)})(1-q^{-2}e^{-(\varepsilon_1+\varepsilon_i)}).
\end{align*}

For each type above, we realise the rank $n-1$ root system as the subsystem obtained by removing the simple root $\alpha_1 = \varepsilon_1 - \varepsilon_2$ from the Dynkin diagram of the rank $n$ root system.   Under this realisation, the corresponding Weyl group $W_{n-1}$ is the subgroup of $W_n$ that fixes $\varepsilon_{\pm 1}$. Let $\rho_n$ denote the half-sum of positive roots of the rank $n$ system, and $\rho_{n-1}$ that of the embedded rank $n-1$ subsystem. Then $\rho_n = \kappa_n\,\varepsilon_1 + \rho_{n-1},$ where $\kappa_n= n-\frac{1}{2}$ for type $\mathsf{B}_n$, $\kappa_n=n$ for type $\mathsf{C}_n$, and $\kappa_n= n-1$ for type $\mathsf{D}_n$. 

\begin{proposition}\label{prop: Combid}
For integers $k\geq 0$, we have the following identities:
\begin{alignat*}{2}
   \qquad \qquad  \qquad  \Delta_n \mathbf{Ch}\, G_{n,k} &= q^{-k}\Delta_n + q^{c_n-1}\mathbf{A}_n(H_{n,k}), 
   &\qquad &\text{type $\mathsf{B}_n$ ($n\geq 2$)},\\
   \qquad   \qquad  \qquad  \Delta_n \mathbf{Ch}\, G_{n,k} &= q^{c_n-1}\mathbf{A}_n(H_{n,k}),
   &\qquad &\text{types $\mathsf{C}_n$ ($n\geq 3$) and $\mathsf{D}_n$ ($n\geq 4$)}.
\end{alignat*}  
where $\mathbf{A}_n$ is the antisymmetriser for the Weyl group $W_n$ of type $\mathsf{B}_n$, $\mathsf{C}_n$, or $\mathsf{D}_n$,  $\Delta_n$ is the corresponding  Weyl denominator for that type, and $c_n$ is defined as in \propref{prop: eigval}.  
\end{proposition}
\begin{proof}
 We  note that 
\[
    \Delta_n= \prod_{\alpha\in \Phi^+\setminus \Phi_{\varepsilon_1}^+}(e^{\frac{\alpha}{2}}- e^{-\frac{\alpha}{2}}) \prod_{\alpha\in \Phi_{\varepsilon_1}^+}(e^{\frac{\alpha}{2}}- e^{-\frac{\alpha}{2}})=\Delta_{n-1} \prod_{\alpha\in \Phi_{\varepsilon_1}^+}(e^{\frac{\alpha}{2}}- e^{-\frac{\alpha}{2}}),
\]
and $w(\Delta_n)= \mathbf{sgn}(w)\Delta_n$ for any $w\in W_n$. Recalling that $R_n$ is the set of left coset representatives for $W_n/W_{n-1}$, we split the sum over the Weyl group $W_n$ accordingly and rewrite $\mathbf{A}_n(H_{n,k})$ as follows: 
\begin{align*}\allowdisplaybreaks
\mathbf{A}_n(H_{n,k})&= \sum_{w\in W_n}\mathbf{sgn}(w)w \Big(e^{(k+\kappa_n)\varepsilon_1+ \rho_{n-1}} \prod_{\alpha\in \Phi_{\varepsilon_1}^+}(1- q^{-2(\alpha, \varepsilon_1)}e^{-\alpha}) \Big)\\ 
&= \sum_{\sigma\in R_n} \mathbf{sgn}(\sigma)  \sigma \bigg( \sum_{u\in W_{n-1}} \mathbf{sgn}(u) u \Big(e^{(k+\kappa_n)\varepsilon_1+ \rho_{n-1}} \prod_{\alpha\in \Phi_{\varepsilon_1}^+}(1- q^{-2(\alpha, \varepsilon_1)}e^{-\alpha})\Big)\bigg)\\ 
&= \sum_{\sigma \in R_n} \mathbf{sgn}(\sigma) \sigma \Big( e^{(k+\kappa_n)\varepsilon_1} \prod_{\alpha\in \Phi_{\varepsilon_1}^+}(1- q^{-2(\alpha, \varepsilon_1)}e^{-\alpha}) \sum_{u\in W_{n-1}} \mathbf{sgn}(u) u (e^{\rho_{n-1}} ) \Big)\\
&= \sum_{\sigma \in R_n} \mathbf{sgn}(\sigma) \sigma \Big(  \Delta_{n-1}  e^{(k+\kappa_n)\varepsilon_1} \prod_{\alpha\in \Phi_{\varepsilon_1}^+}(1- q^{-2(\alpha, \varepsilon_1)}e^{-\alpha})\Big). \\ 
&=  \sum_{\sigma \in R_n} \mathbf{sgn}(\sigma) \sigma \Big(\Delta_n e^{(k+\kappa_n)\varepsilon_1} \prod_{\alpha\in \Phi_{\varepsilon_1}^+} \frac{1- q^{-2(\alpha, \varepsilon_1)}e^{-\alpha}}{e^{\frac{\alpha}{2}}-  e^{-\frac{\alpha}{2}}} \Big)\\ 
&= \Delta_n \sum_{\sigma \in R_n}  \sigma \Big(e^{(k+\kappa_n)\varepsilon_1} \prod_{\alpha\in \Phi_{\varepsilon_1}^+} \frac{1- q^{-2(\alpha, \varepsilon_1)}e^{-\alpha}}{e^{\frac{\alpha}{2}}-  e^{-\frac{\alpha}{2}}} \Big). 
\end{align*}

For type $\mathsf{B}_n$, we substitute $\kappa_n = n - \frac{1}{2}$ and $\Phi_{\varepsilon_1}^+ = \{\varepsilon_1\} \cup \{\varepsilon_1 \pm \varepsilon_i \mid 2 \le i \le n\}$ in the last line,  and rewrite the denominator using the identities
\[e^{\frac{\varepsilon_1}{2}}-  e^{-\frac{\varepsilon_1}{2}}=\frac{(e^{\varepsilon_1}-  e^{-\varepsilon_1})} {(e^{\frac{\varepsilon_1}{2}}+e^{-\frac{\varepsilon_1}{2}})}, \quad  e^{\frac{\varepsilon_1\pm \varepsilon_i}{2}}-  e^{-\frac{\varepsilon_1\pm\varepsilon_i}{2}}=e^{- \frac{\varepsilon_1\mp \varepsilon_i}{2}}(e^{\varepsilon_1}- e^{\mp\varepsilon_i}).  \]
Using $c_n=2n$ for type $\mathsf{B}_n$, we obtain 
\begin{align*}
& q^{c_n-1}\sum_{\sigma \in R_n}  \sigma \Big(e^{(k+\kappa_n)\varepsilon_1} \prod_{\alpha\in \Phi_{\varepsilon_1}^+} \frac{1- q^{-2(\alpha, \varepsilon_1)}e^{-\alpha}}{e^{\frac{\alpha}{2}}-  e^{-\frac{\alpha}{2}}} \Big) \\ 
&= \sum_{\sigma \in R_n}  \sigma \bigg( e^{k\varepsilon_1} \frac{q-q^{-1}+  q e^{\varepsilon_1}- q^{-1} e^{\varepsilon_1}}{e^{\varepsilon_1}- e^{-\varepsilon_1}} \prod_{i=2}^n \frac{(qe^{\varepsilon_1} - q^{-1}e^{\varepsilon_i})(qe^{\varepsilon_1} - q^{-1}e^{-\varepsilon_i})}{ (e^{\varepsilon_1}- e^{\varepsilon_i})(e^{\varepsilon_1}- e^{-\varepsilon_i}) } \bigg)\\ 
&=\mathbf{Ch}\, G_{n,k}- q^{-k}.
\end{align*}
Therefore, we have  $\Delta_n \mathbf{Ch}\, G_{n,k}= q^{-k} \Delta_n + q^{c_n-1} \mathbf{A}_n(H_{n,k})$ in type $\mathsf{B}_n$.  For types $\mathsf{C}_n$ and $\mathsf{D}_n$, an analogous analysis yields the corresponding identities. 
\end{proof}

To connect with representation theory, we recall the following standard results from \cite{FH13}.

\begin{lemma}\label{lem: extpower}
 Let $V$ be the natural module for the Lie algebra $\mathfrak{g}$ of type $\mathsf{B}_n$, $\mathsf{C}_n$ or $\mathsf{D}_n$. In the following, we set $\varpi_0:=0$. 
 \begin{enumerate}
 \item In type $\mathsf{B}_n\,(n\geq 2)$, the exterior powers satisfy
\[ \bigw^{r} V\cong L(\varpi_r), \quad 0\leq r\leq n-1, \quad \bigw^{n} V\cong L(2\varpi_n),  \]
where $\varpi_r= \varepsilon_1+\cdots +\varepsilon_r$ for $1\leq r\leq n-1$ and $2\varpi_n=\varepsilon_1+\cdots +\varepsilon_n$. Moreover, $ \bigw^{r} V \cong \bigw^{2n+1-r}V $ for $n<r\leq 2n+1$. 
 \item In type $\mathsf{C}_n\,(n\geq 3)$, the exterior power $\bigw^{r} V$ admits the multiplicity-free decomposition
\[ \bigw^{r} V \cong \bigoplus_{s=0}^{\lfloor r/2\rfloor }L(\varpi_{r-2s}), \quad 0\leq r\leq n,   \]
 where $\varpi_r= \varepsilon_1+ \cdots +\varepsilon_r$ for $1\leq r\leq n$. Moreover,  $ \bigw^{r} V \cong \bigw^{2n-r}V $ for $n<r\leq 2n$. 

\item In type $\mathsf{D}_n\,(n\geq 4)$, the exterior powers satisfy
 \begin{align*}
  &\bigw^{r}V \cong L(\varpi_r), \quad 0\leq r\leq n-2,\quad \bigw^{n-1}V\cong L(\varpi_{n-1}+\varpi_n), \\ 
 &\bigw^n V\cong L(2\varpi_{n-1})\oplus L(2\varpi_{n}),  
 \end{align*}
 where $\varpi_r= \varepsilon_1+\cdots +\varepsilon_r$ for $1\leq r\leq n-2$,  $2\varpi_{n-1}=\varepsilon_1+\cdots +\varepsilon_{n-1}-\varepsilon_n$, and  $2\varpi_{n}=\varepsilon_1+\cdots +\varepsilon_{n-1}+\varepsilon_n$. Moreover, $ \bigw^{r} V \cong \bigw^{2n-r}V $ for $n<r\leq 2n$. 
 \end{enumerate}
\end{lemma}

\begin{lemma}\label{lem: genfun} Maintain notation as in \lemref{lem: extpower}. The generating function of the characters $\mathbf{Ch}(\bigw^rV)$  of the exterior powers is given as follows. 
 \begin{enumerate}
   \item In type $\mathsf{B}_n\, (n\geq 2)$, we have 
          \[ \sum_{r=0}^{2n+1} {\bf Ch}(\bigw^r V)t^r= (1+t)\prod_{i=1}^n (1+te^{\varepsilon_i})(1+te^{-\varepsilon_i}).   \]
   \item In both types $\mathsf{C}_n\, (n\geq 3)$ and $\mathsf{D}_n\,(n\geq 4)$, we have 
   \[ \sum_{r=0}^{2n} {\bf Ch}(\bigw^rV)t^r= \prod_{i=1}^n (1+te^{\varepsilon_i})(1+te^{-\varepsilon_i}).   \]
 \end{enumerate}
\end{lemma}

\subsection{Representation-theoretic interpretation}
We now reformulate the elements $G_{n,k}$ in terms of characters of finite-dimensional representations of $\fg$.
\subsubsection{Type $\mathsf{B}_n$} We introduce some notation to facilitate the statement of the main result. In the case of type $\mathsf{B}_n\,(n\geq 2)$, for $0\leq r\leq 2n-1$, we define $\bar{r}$ and $\omega_{\bar{r}}$ as follows: 
\begin{align*}
\text{Type $\mathsf{B}_n$:} \quad 
& \bar{r} := \min\{r, 2n - 1 - r\} \in \{0, 1, \dots, n - 1\}, \\
& \mu_0 := 0, \quad 
  \mu_{\bar{r}} := \varepsilon_2 + \cdots + \varepsilon_{\bar{r} + 1}, 
  \quad \text{for } 1 \leq \bar{r} \leq n - 1.
\end{align*}
For $k\geq 0$ and $0\leq r\leq 2n-1$, we define $\lambda_{k}^r$ by 
\[  \lambda_k^r:=(k-r)\varepsilon_1+ \mu_{\bar{r}}.  \] 
When $k-r\geq 1$, $\lambda_k^r$ is a dominant integral weight of type $\mathsf{B}_n$ and can be represented  by a Young diagram of hook shape. 

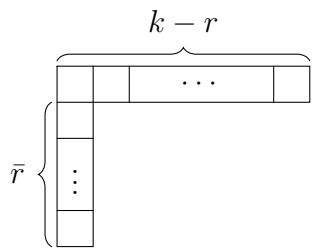
\begin{figure}[ht]
\begin{tikzpicture}[x=6mm,y=6mm,line cap=round, scale=0.8] 
  \def\Arm{7} % visual arm extent (horizontal)
  \def\Leg{5} % visual leg extent (vertical)

  % --- Outer hook, left-aligned ---
  \draw (0,0) -- (\Arm,0) -- (\Arm,-1) -- (1,-1) -- (1,-\Leg) -- (0,-\Leg) -- (0,0);

  % --- A few internal ticks to suggest boxes (not counted) ---                
  \draw (1,0) -- (1,-1);                 % corner column
  \draw (2,0) -- (2,-1);                 % a tick in the arm
  \draw (\Arm-1,0) -- (\Arm-1,-1);       % a tick near the right end\
  \draw (0,-1) -- (1,-1);   
  \draw (0,-2) -- (1,-2);                % ticks in the leg
  \draw (0,-4) -- (1,-4);

  % --- Ellipses matching your sketch ---
  \node at ({0.5*\Arm+0.5},-0.5) {$\cdots$};
  \node at (0.5,-3) {$\vdots$};

  % --- Braces and labels ---
  \draw[decorate,decoration={brace,amplitude=5pt}]
    (0,0.15) -- (\Arm,0.15) node[midway,above=6pt] {$k-r$};
  \draw[decorate,decoration={brace,amplitude=5pt,mirror}]
    (-0.15,-1) -- (-0.15,-\Leg) node[midway,left=6pt] {$\bar r$};
\end{tikzpicture}
\caption{The hook partition $\lambda_k^r$.}
\label{fig: youngdia}
\end{figure}

\begin{theorem}\label{thm: ChGtypeB}
Maintain notation as above. In type $\mathsf{B}_n\,(n\geq 2)$, for all integers  $k\geq 0$, we have 
\[
\mathbf{Ch}\,G_{n,k} = 
\begin{cases}
1 + \displaystyle\sum_{r=0}^{2n-1} q^{2n-2r-1}, & \text{if } k = 0, \\
\displaystyle\sum_{r=0}^{k-1} (-1)^r q^{2n - 2r - 1} \chi(\lambda_k^r), & \text{if } 1 \leq k \leq 2n - 1 \text{ and } k \text{ is odd}, \\
q^{-k} + \displaystyle\sum_{r=0}^{\min\{k-1, 2n-1\}} (-1)^r q^{2n - 2r - 1} \chi(\lambda_k^r), & \text{otherwise,}
\end{cases}
\]
where $\chi(\lambda_k^r)$ denotes the character of the simple module $L(\lambda_k^r)$ with highest weight $\lambda_k^r$. 
\end{theorem}
\begin{proof}
 Since $\mathbf{Ch}\, G_{n,k}$ and $\mathbf{A}_{n}(H_{n,k})$ are related by the identity in \propref{prop: Combid}, we first compute  $\mathbf{A}_{n}(H_{n,k})$. Let $U$ be the natural module of $\mathfrak{so}_{2n-1}$ (viewed as the type $\mathsf{B}_{n-1}$ subalgebra embedded in the type $\mathsf{B}_n$), with  weights   $\varepsilon_0$ and  $\varepsilon_{\pm i}, 2\leq i\leq n$.   Using the generating function from \lemref{lem: genfun} with $t = -q^{-2}e^{-\varepsilon_1}$ for $\mathfrak{so}_{2n-1}$, the expression for $H_{n,k}$ can be rewritten as follows:
\begin{align*}
\mathbf{A}_{n}(H_{n,k}) &= \mathbf{A}_n \Big(e^{\rho_n+k \varepsilon_1} (1+t)\prod_{i=2}^n (1+t e^{\varepsilon_i})(1+t e^{-\varepsilon_i})\Big) = \sum_{r=0}^{2n-1}\mathbf{A}_n \Big(e^{\rho_n+k \varepsilon_1}  {\bf Ch}(\bigw^r U)t^r  \Big) \\
&=  \sum_{r=0}^{2n-1} (-q^{-2})^r \mathbf{A}_n \Big(e^{\rho_n+(k-r) \varepsilon_1}  {\bf Ch}(\bigw^r U)\Big). 
\end{align*}

Using part (1) of \lemref{lem: extpower} for $\mathfrak{so}_{2n-1}$, for $0\leq r\leq 2n-1$, we have  $\bigw^rU \cong \bigw^{2n-1-r}U \cong L(\mu_{\bar{r}})$  as $\mathfrak{so}_{2n-1}$-modules.
Hence, for each $r=0,1, \dots,  2n-1$, upon splitting the alternating sum over the Weyl group $W_n$ and  invoking  \lemref{lem: altiden},  we have 
\begin{align*}
\mathbf{A}_n  \Big(e^{\rho_n+(k-r) \varepsilon_1}  {\bf Ch}(\bigw^r U)\Big) 
&= \sum_{\sigma\in R_n}  \mathbf{sgn}(\sigma) \sigma (e^{(n+k-r-\frac{1}{2})\varepsilon_1} )  \sum_{u\in W_{n-1}} \mathbf{sgn}(u)u (e^{\rho_{n-1}} \chi(\mu_{\bar{r}})  )  \\
&= \sum_{\sigma\in R_n}  \mathbf{sgn}(\sigma) \sigma (e^{(n+k-r-\frac{1}{2})\varepsilon_1} )  \sum_{u\in W_{n-1}} \mathbf{sgn}(u)u (e^{\rho_{n-1}+ \mu_{\bar{r}}})\\ 
&=\mathbf{A}_n  (e^{(k-r)\varepsilon_1+ \mu_{\bar{r}} +\rho_n})= \mathbf{A}_n  (e^{\lambda_k^r+ \rho_n }). 
\end{align*}
Combining this with \propref{prop: Combid}, we obtain
\[
    \Delta_n \mathbf{Ch}\, G_{n,k} = q^{-k}\Delta_n +  q^{2n-1} \sum_{r=0}^{2n-1} (-q^{-2})^r \mathbf{A}_n  (e^{\lambda_k^r + \rho_n}).
\]
It remains to determine the values of $r$ for which $\mathbf{A}_n(e^{\rho_n+\lambda_k^r}) \neq 0$; for these cases, dividing by $\Delta_n$ and applying the Weyl character formula yields the desired expression of  $\mathbf{Ch}\,G_{n,k}$ in terms of irreducible characters.

{\bf Case 1}: $0\leq r\leq k-1$.  Then  $k\geq 1$,  and  $\lambda^r_k$ is a  dominant integral weight. By the Weyl character formula, we have
\[ \mathbf{A}_n (e^{\lambda_k^r+ \rho_n})= \Delta_n \chi(\lambda_k^r), \quad 0\leq r \leq \min\{k-1, 2n-1\}.  \]

{\bf Case 2}: $k\leq r\leq 2n-1$.   We analyse this case by looking  at the coordinates of  $\lambda_k^r+\rho_n$. For $1\leq i\leq n$, let $c_i$ denote the coefficient of $\varepsilon_i$  of $\lambda_k^r+\rho_n$. Then,  for $0\leq r\leq 2n-1$, we have 
\begin{align*}
 &c_1 = n-\frac{1}{2}+k-r, \\ 
 &S_1:=\{c_2,c_3 \dots ,c_{\bar{r}+1}\}=\{n-\frac{1}{2}, n-\frac{3}{2},  \dots, n-\bar{r}+\frac{1}{2}\}, \\ 
 &S_2:=\{c_{\bar{r}+2}, \dots, c_n\}= \{n-\bar{r}-\frac{3}{2},n-\bar{r}-\frac{5}{2}, \dots, \frac{1}{2}\}, 
\end{align*}
where $\bar{r}= \min\{r,2n-1-r\}$.  Conventionally, we set $S_1:=\emptyset$ when $\bar{r}=0$ and $S_2:=\emptyset$ when $\bar{r}=n-1$. Thus, for $0\leq r \leq 2n-1$,  the union $S_1\cup S_2$ covers every half integer from $\frac{1}{2}$ up to $n-\frac{1}{2}$ except for the missing number $n-\bar{r}-\frac{1}{2}$. On the other hand, since $k\leq r\leq 2n-1$, we have  $c_1\in \{\pm \frac{1}{2},\pm \frac{3}{2}, \dots, \pm(n-\frac{1}{2})\}$, so exactly one of the following occurs:
\begin{enumerate}
\item [(C1)] If $\pm c_1= c_i$ for some $c_i\in S_1\cup S_2$, then $(\lambda_k^r+ \rho_n, \varepsilon_1 \mp \varepsilon_i)=0$. It then follows from \lemref{lem: altvanish} that $\mathbf{A}_n(e^{\lambda_k^r+ \rho_n})=0$;
\item [ (C2)] If $\pm c_1=n-\bar{r}-\frac{1}{2}$ (the missing number), then  $\rho_n+\lambda_k^r$ is equal to $\rho_n$ up to a permutation  of coordinates in $W_n$. This permutation is obtained  by (optionally) applying  the reflection $s_{\varepsilon_1}$ when $-c_1=n-\bar{r}-\frac{1}{2}$, and  then applying successive adjacent reflections $s_{\varepsilon_i-\varepsilon_{i+1}}$ for $i=1,\dots,\bar{r}$, each contributing a factor of $-1$ to the sign character. Thus, if $\pm c_1=n-\bar{r}-\frac{1}{2}$, we have
\[ \mathbf{A}_n(e^{\lambda_k^r+ \rho_n})=\pm (-1)^{\bar{r}}\mathbf{A}_n(e^{\rho_n}). \]
\end{enumerate}
Therefore, $\mathbf{A}_n(e^{\lambda_k^r+ \rho_n})$ is nonzero precisely when (C2) holds. We now have two  subcases according to  the value of $k$. 

{\bf Subcase 2.1:} $k=0$, so we have $c_1= n-\frac{1}{2}-r$. If  $0\leq r\leq n-1$, then  $\bar{r} = r$, and $c_1 = n-\bar{r}-\frac{1}{2}$, whence (C2) is satisfied. If $n\leq r\leq 2n-1$, then  $\bar{r} = 2n-1-r$, and $-c_1 = r-n+\frac{1}{2} = n-\bar{r}-\frac{1}{2}$, whence (C2) is again satisfied. Therefore, we have 
\[  \mathbf{A}_n(e^{\lambda_k^r+ \rho_n})=(-1)^r \mathbf{A}_n(e^{\rho_n}), \quad k=0\leq r\leq 2n-1. \]

{\bf Subcase 2.2:} $k\geq 1$. If $ r\leq n-1$, then $\bar{r}=r$. In this case, (C2) cannot hold, since either condition $c_1=n-\bar{r}-\tfrac{1}{2}$ or $-c_1=n-\bar{r}-\tfrac{1}{2}$ would contradict the assumption $k\geq 1$. Thus only (C1) is satisfied, and
\[
\mathbf{A}_n(e^{\lambda_k^r+ \rho_n})=0, \quad 1\leq k\leq r\leq n-1.
\]
If $n\leq r\leq 2n-1$, then $\bar{r} = 2n-1-r$.  The condition $-c_1=n-\bar{r}-\frac{1}{2}$ leads to $k=0$, which contradicts the assumption $k\geq 1$. The condition $c_1=n-\bar{r}-\frac{1}{2}$ leads to $r=n+\frac{k-1}{2}$, which forces $k$ to be an odd integer.

 Thus, for $1\leq k\leq r\leq 2n-1$, (C2) holds if and only if $k$ is odd and $r=n+\frac{k-1}{2}$. In this case, we have 
\[ \mathbf{A}_n(e^{\lambda_k^r+ \rho_n})= (-1)^{\bar{r}}\mathbf{A}_n(e^{\rho_n})= (-1)^{n-\frac{k+1}{2}}\mathbf{A}_n(e^{\rho_n}). \]

Summarising the above case analysis according to the range of $k$, and applying the Weyl denominator formula $\Delta_n = \mathbf{A}_n(e^{\rho_n})$ and the Weyl character formula $\chi(\lambda_k^r)= \mathbf{A}_n  (e^{\lambda_k^r + \rho_n})/ \Delta_n$, we obtain the three stated expressions for $\mathbf{Ch}\,G_{n,k}$ in terms of characters. 
\end{proof}

\subsubsection{Type $\mathsf{C}_n$}
We proceed to address type $\mathsf{C}_n\,(n\geq 3)$. Similarly, we introduce the following notation. For $0\leq r\leq 2n$, we define 
\begin{align*}
\text{Type $\mathsf{C}_n$:} \quad 
& \bar{r} := \min\{r, 2n - r, n-1\} \in \{0, 1, \dots, n - 1\}, \\
& \mu_0 := 0, \quad 
  \mu_{\bar{r}} := \varepsilon_2 + \cdots + \varepsilon_{\bar{r} + 1}, 
  \quad \text{for } 1 \leq \bar{r} \leq n - 1.
\end{align*}
For $k\geq 0$ and $0\leq r\leq 2n$, we define weights $\lambda_{k}^r$ and signs $\tau_r$ by 
\begin{align*}
\lambda_k^r&:=(k-r)\varepsilon_1+ \mu_{\bar{r}},\\ 
\tau_r&:= \begin{cases}
 1, \quad &\text{if } 0\leq r\leq n-1, \\ 
 0, \quad &\text{if } r=n,\\ 
 -1, \quad &\text{if } n+1\leq r\leq 2n.  
\end{cases}
\end{align*} 
When $k-r\geq 1$, $\lambda_k^r$ is a dominant integral weight of type $\mathsf{C}_n$ represented by the  Young diagram  in \figref{fig: youngdia}. The sign $\tau_r$ indicates  whether $\lambda_k^r$ contributes positively, negatively, or not at all to the character expression of $\mathbf{Ch}\,G_{n,k}$; in particular, although $\lambda_k^n$ is defined, it contributes nothing.

\begin{theorem}\label{thm: ChGtypeC}
Maintain notation as above. In type $\mathsf{C}_n\,(n\geq 3)$, for all integers  $k\geq 0$, we have 
\[
\mathbf{Ch}\,G_{n,k} = 
\begin{cases}
 \displaystyle\sum_{r=0}^{n-1}q^{2n-2r}+ \sum_{r=n+1}^{2n}q^{2n-2r}, & \text{if } k = 0, \\
-q^{-k}+\displaystyle\sum_{r=0}^{k-1}(-1)^r\tau_r q^{2n-2r} \chi(\lambda_k^r), & \text{if } 1 \leq k \leq 2n  \text{ and } k \text{ is even}, \\
\displaystyle\sum_{r=0}^{\min\{k-1,2n\}} (-1)^{r}\tau_r q^{2n-2r} \chi(\lambda_k^r), & \text{otherwise,}
\end{cases}
\]
where $\chi(\lambda_k^r)$ denotes the character of the simple module $L(\lambda_k^r)$ with highest weight $\lambda_k^r$. 
\end{theorem}

\begin{proof}
The approach is similar to that of \thmref{thm: ChGtypeB}. The only difference is the representation theory. In accordance with the embedding of $\mathfrak{sp}_{2n-2}$ into $\mathfrak{sp}_{2n}$,  let $U$ be the natural module of $\mathfrak{sp}_{2n-2}$ whose  weights are  $\varepsilon_{\pm i}, 2\leq i\leq n$. Using  the generating function from \lemref{lem: genfun} for $\mathfrak{sp}_{2n-2}$ with $t = -q^{-2}e^{-\varepsilon_1}$, the expression for $H_{n,k}$ of type $\mathsf{C}_n$ can be rewritten as follows:
\begin{align*}
\mathbf{A}_{n}(H_{n,k}) &= \mathbf{A}_n \Big(e^{\rho_n+k \varepsilon_1} (1-t^2)\prod_{i=2}^n (1+t e^{\varepsilon_i})(1+t e^{-\varepsilon_i})\Big) = \sum_{r=0}^{2n-2}\mathbf{A}_n \Big(e^{\rho_n+k \varepsilon_1} (1-t^2) {\bf Ch}(\bigw^r U)t^r  \Big) \\
&=  \sum_{r=0}^{2n} \mathbf{A}_n \Big(e^{\rho_n+k \varepsilon_1}  ({\bf Ch}(\bigw^r U)- {\bf Ch}(\bigw^{r-2} U))t^r\Big),
\end{align*}
where $\mathbf{Ch}(\bigw^{-1}U)=\mathbf{Ch}(\bigw^{-2}U):=0$ and note also that $\mathbf{Ch}(\bigw^{2n-1}U)=\mathbf{Ch}(\bigw^{2n}U)=0$. Using part (2) of \lemref{lem: extpower} for $\mathfrak{sp}_{2n-2}$, we have 
\[
{\bf Ch}(\bigw^r U)- {\bf Ch}(\bigw^{r-2} U)= 
\begin{cases}
\chi(\mu_{r}), \quad &0\leq r\leq n-1\\ 
0, \quad  &r=n, \\ 
-\chi(\mu_{2n-r}), \quad & n+1\leq r\leq 2n.
\end{cases}
\]
Substituting the above expressions and $t = -q^{-2}e^{-\varepsilon_1}$ back into the previous sum, we obtain
\begin{align*}
\mathbf{A}_n(H_{n,k})&= \sum_{r=0}^{n-1}(-q^{-2})^r \mathbf{A}_n(e^{\rho_n+ (k-r)\varepsilon_1} \chi(\mu_r)) -\sum_{r=n+1}^{2n} (-q^{-2})^r\mathbf{A}_n(e^{\rho_n+(k-r)\varepsilon_1} \chi(\mu_{2n-r}))\\
&=  \sum_{r=0}^{n-1}(-q^{-2})^r \mathbf{A}_n(e^{\rho_n+ (k-r)\varepsilon_1+\mu_r}) -\sum_{r=n+1}^{2n} (-q^{-2})^r\mathbf{A}_n(e^{\rho_n+(k-r)\varepsilon_1+\mu_{2n-r}} ),\\ 
&= \sum_{r=0}^{2n} (-q^{-2})^r \tau_r \mathbf{A}_n(e^{\rho_n+ \lambda_k^r}),
\end{align*}
where the second equality follows by the same reasoning as in the type $\mathsf{B}_n$ case (namely, by splitting the alternating sum over coset representatives $R_n$ and $W_{n-1}$, and applying Lemma~\ref{lem: altiden}). Combing with \propref{prop: Combid}, we obtain
\[\Delta_n {\bf Ch}\, G_{n,k}= q^{2n}\sum_{r=0}^{2n} (-q^{-2})^r \tau_r \mathbf{A}_n(e^{\rho_n+ \lambda_k^r}).\]
It remains to determine the indices $r$ for which $\mathbf{A}_n(e^{\rho_n+\lambda_k^r})\neq 0$; the case $r=n$ is excluded since $\tau_n=0$. The analysis below is similar to the type $\mathsf{B}_n$ case. 

{\bf Case 1}:  $0\leq r\leq k-1$. Then $k\geq 1$, and  $\lambda^r_k$ is a  dominant integral weight. By the Weyl character formula, 
\[ \mathbf{A}_n (e^{\lambda_k^r+ \rho_n})= \Delta_n \chi(\lambda_k^r), \quad 0\leq r\leq \min\{k-1,2n\}.  \]

{\bf Case 2}: $k\leq r\leq 2n$.   We analyse this case by looking  at the coordinates of  $\lambda_k^r+\rho_n$. For $1\leq i\leq n$, let $c_i$ denote the coefficient of $\varepsilon_i$  of $\lambda_k^r+\rho_n$. Then,  for $0\leq r\leq 2n$ ($r\neq n$), we have 
\begin{align*}
 &c_1 = n+k-r, \\ 
 &S_1:=\{c_2,c_3 \dots ,c_{\bar{r}+1}\}=\{n, n-1,  \dots, n-\bar{r}+1\}, \\ 
 &S_2:=\{c_{\bar{r}+2}, \dots, c_n\}= \{n-\bar{r}-1,n-\bar{r}-2, \dots, 1\},  
\end{align*}
where $\bar{r}= \min\{r,2n-r,n-1\}$. 
 Conventionally, we set $S_1:=\emptyset$ when $\bar{r}=0$ and $S_2:=\emptyset$ when $\bar{r}=n-1$. For $0\leq r\leq 2n\ (r\neq n)$, we note that $S_1\cup S_2$ covers every integer from $1$ up to $n$ except for the missing number $n-\bar{r}$. On the other hand, since $k\leq r\leq 2n$, we have  $c_1\in \{0, \pm 1, \pm 2, \dots, \pm n\}$, so exactly one of the following occurs:
\begin{enumerate}
\item [(C1)] If $c_1=0$, then we have $(\lambda_k^r+\rho_n, 2\varepsilon_1)=0$, and hence by \lemref{lem: altvanish},  $\mathbf{A}_n(e^{\lambda_k^r+ \rho_n})=0$. 
\item [(C2)] If $\pm c_1= c_i$ for some $c_i\in S_1\cup S_2$ , then $(\lambda_k^r+ \rho_n, \varepsilon_1 \mp \varepsilon_i)=0$. It then follows from \lemref{lem: altvanish} that $\mathbf{A}_n(e^{\lambda_k^r+ \rho_n})=0$.
\item [ (C3)] If $\pm c_1=n-\bar{r}$ (the missing number), then  $\rho_n+\lambda_k^r$ is equal to $\rho_n$ up to a permutation  of coordinates in $W_n$. This permutation is obtained similarly as in type $\mathsf{B}$. Thus,  we have
\[ \mathbf{A}_n(e^{\lambda_k^r+ \rho_n})=\pm (-1)^{\bar{r}}\mathbf{A}_n(e^{\rho_n}). \]
\end{enumerate}
Therefore, $\mathbf{A}_n(e^{\lambda_k^r+ \rho_n})$ is nonzero precisely when (C3) holds. We now have two  subcases according to  the value of $k$. 

{\bf Subcase 2.1:} $k=0$, so we have $c_1= n-r$. If  $0\leq r\leq n-1$, then  $\bar{r} = r$, and $c_1 = n-\bar{r}$, whence (C3) is satisfied. If $n< r\leq 2n$, then  $\bar{r} = 2n-r$, and $-c_1 = n-\bar{r}$, whence (C3) is again satisfied. Therefore, we have 
\[  \mathbf{A}_n(e^{\lambda_k^r+ \rho_n})=
\begin{cases}
(-1)^r \mathbf{A}_n(e^{\rho_n}), \quad &0\leq r\leq n-1, \\ 
(-1)^{r+1} \mathbf{A}_n(e^{\rho_n}), \quad &n+1\leq r\leq 2n. 
\end{cases} \]

{\bf Subcase 2.2:} $k\geq 1$. If $ r\leq n-1$, then $\bar{r}=r$. In this case, (C3) cannot hold, since $\pm c_1=n-\bar{r}$  would  contradict the assumption $k\geq 1$. Thus, we have
\[
\mathbf{A}_n(e^{\lambda_k^r+ \rho_n})=0, \quad 1\leq k\leq r\leq n-1.
\]
If $n< r\leq 2n$, then $\bar{r} = 2n-r$.  The condition $-c_1=n-\bar{r}$ leads to $k=0$, which contradicts the assumption $k\geq 1$. The condition $c_1=n-\bar{r}$ leads to $r=n+\frac{k}{2}$, which forces $k$ to be an even integer.

 Thus, for $1\leq k\leq r\leq 2n \ (r\neq n)$, (C3) holds if and only if $k$ is even and $r=n+\frac{k}{2}$. In this case, we have 
\[ \mathbf{A}_n(e^{\lambda_k^r+ \rho_n})= (-1)^{\bar{r}}\mathbf{A}_n(e^{\rho_n})= (-1)^{n-\frac{k}{2}}\mathbf{A}_n(e^{\rho_n}). \]

Summarising the above case analysis according to the range of $k$, and applying the Weyl denominator formula $\Delta_n = \mathbf{A}_n(e^{\rho_n})$ and the Weyl character formula $\chi(\lambda_k^r)= \mathbf{A}_n  (e^{\lambda_k^r + \rho_n})/ \Delta_n$, we obtain the three stated expressions for $\mathbf{Ch}\,G_{n,k}$ in terms of characters. 
\end{proof}

\subsubsection{Type $\mathsf{D}_n$}
Finally, we follow the same approach to address type $\mathsf{D}_n\,(n\geq 4)$. For $0 \leq r \leq 2n-2$, we define $\bar{r}$ and the associated weights $\mu_{\bar{r}}$ as follows:
\begin{align*}
\text{Type } \mathsf{D}_n: \quad
&\bar{r} := \min\{r, 2n - 2 - r\} \in \{0,1,\dots,n-1\}, \\ 
&\mu_0 := 0, \qquad
\mu_{\bar{r}} := \varepsilon_2 + \cdots + \varepsilon_{\bar{r}+1}, \quad 1 \leq \bar{r} \leq n-1,\\ 
&\bar{\mu}_{n-1} := \varepsilon_2 + \cdots + \varepsilon_{n-1} - \varepsilon_{n}. 
%&\mu_{n-1} := \varepsilon_2 + \cdots + \varepsilon_{n-1} + \varepsilon_{n}.
\end{align*}
For $k \geq 0$, we define
\begin{align*}
&\lambda_k^r := (k-r)\varepsilon_1 + \mu_{\bar{r}}, \qquad 0 \leq r \leq 2n-2,\\ 
 &\bar{\lambda}_k^{n-1} := (k-n+1)\varepsilon_1 + \bar{\mu}_{n-1}.
\end{align*}
If $k-r \geq 1$, then $\lambda_k^r$ is a dominant integral weight of type $\mathsf{D}_n$ (hook shape in \figref{fig: youngdia}). Moreover, for $r=n-1$ and $k\ge n$, the weight $\bar\lambda_k^{\,n-1}$ is also dominant integral.

\begin{theorem}\label{thm: ChGtypeD}
Maintain notation as above.  
In type $\mathsf{D}_n (n \geq 4)$, for all integers $k \geq 0$, we have
\[
\mathbf{Ch}\,G_{n,k} =
\begin{cases}
1 + \displaystyle\sum_{r=0}^{2n-2} q^{2n-2-2r}, 
& k = 0, \\

q^{-k} + \displaystyle\sum_{r=0}^{k-1} (-1)^r q^{2n-2-2r}\,
   \Big(\chi(\lambda_k^r) + \delta_{r,n-1}\, \chi(\bar{\lambda}_k^r)\Big), 
& 1 \leq k \leq 2n-2 \ \text{and $k$ even}, \\

\displaystyle\sum_{r=0}^{\min\{k-1,\,2n-2\}} (-1)^r q^{2n-2-2r}\,
   \Big(\chi(\lambda_k^r) + \delta_{r,n-1}\, \chi(\bar{\lambda}_k^r)\Big), 
& \text{otherwise},
\end{cases}
\]
where $\chi(\lambda_k^r)$ denotes the character of the simple module $L(\lambda_k^r)$ of highest weight $\lambda_k^r$, and $\chi(\bar{\lambda}_k^{\,n-1})$ the character of $L(\bar{\lambda}_k^{\,n-1})$.
\end{theorem}

\begin{proof}
The beginning of the proof is similar to that of \thmref{thm: ChGtypeB}. We outline main steps below.  Let $U$ be the natural module of $\mathfrak{so}_{2n-2}$ whose  weights are  $\varepsilon_{\pm i}, 2\leq i\leq n$. We first rewrite the expression for $H_{n,k}$ of type $\mathsf{D}_n$ using the generating function from \lemref{lem: genfun} with $t = -q^{-2}e^{-\varepsilon_1}$:
\begin{align*}
\mathbf{A}_{n}(H_{n,k}) &= \mathbf{A}_n \Big(e^{\rho_n+k \varepsilon_1} \prod_{i=2}^n (1+t e^{\varepsilon_i})(1+t e^{-\varepsilon_i})\Big) = \sum_{r=0}^{2n-2}\mathbf{A}_n \Big(e^{\rho_n+k \varepsilon_1}  {\bf Ch}(\bigw^r U)t^r  \Big) \\
&=  \sum_{r=0}^{2n-2} (-q^{-2})^r \mathbf{A}_n \Big(e^{\rho_n+(k-r) \varepsilon_1}  {\bf Ch}(\bigw^r U)\Big). 
\end{align*}
By part (3) of \lemref{lem: extpower}, we have
\begin{align*}
&\bigw^rU\cong  \bigw^{2n-2-r}U \cong L(\mu_{\bar{r}}), \quad  \text{$0\leq r\leq 2n-2$ with $r\neq n-1$},  \\ 
&\bigw^{n-1}U \cong  L(\mu_{n-1})\oplus L(\bar{\mu}_{n-1}).
\end{align*}
Substituting the above isomorphisms back into the previous sum, we obtain
\begin{align*}
   \mathbf{A}_{n}(H_{n,k})=&\sum_{r=0}^{2n-2} (-q^{-2})^r \mathbf{A}_n\Big(e^{\rho_n+(k-r) \varepsilon_1}\big(\chi(\mu_{\bar{r}})+ \delta_{r,n-1}\chi(\bar{\mu}_{n-1})\big) \Big) \\ 
   =& \sum_{r=0}^{2n-2} (-q^{-2})^r \mathbf{A}_n(e^{\lambda_k^r+\rho_n}+ \delta_{r,n-1}e^{\bar{\lambda}_k^r+\rho_n})  \\  
\end{align*}
Hence, in view of \propref{prop: Combid}, we have 
\[
\Delta_n {\bf Ch}\, G_{n,k}= q^{2n-2}\sum_{r=0}^{2n-2} (-q^{-2})^r \mathbf{A}_n(e^{\lambda_k^r+\rho_n}+ \delta_{r,n-1}e^{\bar{\lambda}_k^r+\rho_n}).\]
It remains to determine the indices $r$ for which $\mathbf{A}_n(e^{\lambda_k^r+\rho_n})\neq 0$ and $\mathbf{A}_n(e^{\bar{\lambda}_k^{n-1}+\rho_n})\neq 0$. 

We first consider the weight $\bar{\lambda}_k^{n-1}$ in the case $r=n-1$. By definition, $\bar{\lambda}_k^{n-1}$ is  a dominant integral weight whenever $k\ge n$. Hence, by the Weyl character formula, 
\[ \mathbf{A}_n(e^{\bar{\lambda}_k^{n-1}+\rho_n})= \Delta_n\chi(\bar{\lambda}_k^{n-1}), \quad \text{for all } k\geq n. \]
 For $k\leq n-1$,  the coordinates of $\bar{\lambda}_k^{n-1}+\rho_n$ with respect to the basis $\{\varepsilon_i\}_{i=1}^n$ are  
\[\bar{\lambda}_k^{n-1}+\rho_n= (k, n-1, n-2, \cdots, 2, -1).   \]
Thus, if $2\leq k< n$, we have $(\bar{\lambda}_k^{n-1}+\rho_n, \varepsilon_1- \varepsilon_{n-k+1})=0$, and  by \lemref{lem: altvanish} it follows that $\mathbf{A}_n(e^{\bar{\lambda}_k^{n-1}+\rho_n})=0$. For $k=1$, we have  $(\bar{\lambda}_1^{n-1}+\rho_n, \varepsilon_1+\varepsilon_{n})=0$, so again $\mathbf{A}_n(e^{\bar{\lambda}_1^{n-1}+\rho_n})=0$ by \lemref{lem: altvanish}. Finally,  if $k=0$,  we have
$\bar{\lambda}_0^{n-1}+\rho_n=(0, n-1,n-2,\dots,2,-1)$.
Applying the reflection $s_{\varepsilon_1+\varepsilon_n}$, this transforms to $(1,n-1,n-2,\dots,2,0)$.
Successively applying the reflections $s_{\varepsilon_i-\varepsilon_{i+1}}$ for $1\leq i\leq n-2$ then sends this weight to $\rho_n$. Since each reflection contributes a factor of $-1$ to the sign character, we conclude that
\[
    \mathbf{A}_n(e^{\bar{\lambda}_0^{n-1}+\rho_n})=(-1)^{n-1}\mathbf{A}_n(e^{\rho_n}), \quad k=0 \text{ and } r=n-1. 
\]

The remaining analysis for $\lambda_k^r$ for $0\leq r\leq 2n-2$ and $k\geq 0$  proceeds similarly as in type $\mathsf{B}_n$. The difference in Weyl groups introduces subtleties, so we include a complete account of the argument below.

{\bf Case 1}: $0\leq r\leq k-1$. Then $k\geq 1$, and  $\lambda^r_k$ is a  dominant integral weight. By the Weyl character formula, 
\[ \mathbf{A}_n (e^{\lambda_k^r+ \rho_n})= \Delta_n \chi(\lambda_k^r), \quad 0\leq r\leq \min\{k-1, 2n-2\}.  \]

{\bf Case 2}: $k\leq r\leq 2n-2$.   We analyse this case by looking  at the coordinates of  $\lambda_k^r+\rho_n$. For $1\leq i\leq n$, let $c_i$ denote the coefficient of $\varepsilon_i$  of $\lambda_k^r+\rho_n$. Then,  for $0\leq r\leq 2n-2$, we have 
\begin{align*}
 &c_1 = n-1+k-r, \\ 
 &S_1:=\{c_2,c_3 \dots ,c_{\bar{r}+1}\}=\{n-1, n-2,  \dots, n-\bar{r}\}, \\ 
 &S_2:=\{c_{\bar{r}+2}, \dots, c_n\}= \{n-\bar{r}-2,n-\bar{r}-3, \dots, 0\}. 
 \end{align*}
where $\bar{r}= \min\{r,2n-2-r\}$. 
Conventionally, we set $S_1:=\emptyset$ when $\bar{r}=0$ and $S_2:=\emptyset$ when $\bar{r}=n-1$. For $0\leq r\leq 2n-2$, we note that $S_1\cup S_2$ covers every integer from $0$ up to $n-1$ except the missing number $n-\bar{r}-1$. On the other hand, since $k\leq r\leq 2n-2$, we have  $c_1\in \{0, \pm 1, \pm 2, \dots, \pm (n-1)\}$, so exactly one of the following occurs:
\begin{enumerate}
\item [(C1)] If $\pm c_1= c_i$ for some $c_i\in S_1\cup S_2$, then $(\lambda_k^r+ \rho_n, \varepsilon_1 \mp \varepsilon_i)=0$. It then follows from \lemref{lem: altvanish} that $\mathbf{A}_n(e^{\lambda_k^r+ \rho_n})=0$;
\item [ (C2)] If $\pm c_1=n-\bar{r}-1$ (the missing number), then  $\rho_n+\lambda_k^r$ is equal to $\rho_n$ up to a permutation  $\sigma\in W_n$ of coordinates. Thus, we have 
\[ \mathbf{A}_n(e^{\lambda_k^r+ \rho_n})=\mathbf{sgn}(\sigma)\mathbf{A}_n(e^{\rho_n}), \]
where $\sigma$ will be determined explicitly in the discussions below. 
\end{enumerate}
Therefore, $\mathbf{A}_n(e^{\lambda_k^r+ \rho_n})$ is nonzero precisely when (C2) holds. We now have two  subcases according to  the value of $k$. 

{\bf Subcase 2.1:} $k=0$, so we have $c_1= n-1-r$. If  $0\leq r\leq n-1$, then  $\bar{r} = r$, and $c_1 = n-1-\bar{r}$, whence (C2) is satisfied. In this case, the weight $\lambda_k^r+\rho_n$ differs from $\rho_n$ only by permuting the first $r+1$ coordinates. For $1\leq r\leq n-1$,  this permutation can be realised by successive simple reflections $s_{\varepsilon_i-\varepsilon_{i+1}}$, each contributing a factor of $-1$ to the sign character. Consequently, we have
\[\mathbf{A}_n(e^{\lambda_k^r+ \rho_n})=(-1)^{r}\mathbf{A}_n(e^{\rho_n}), \quad k=0 \text{\, and \,} 0\leq r\leq n-1.\]

If $n\leq  r\leq 2n-2$, then  $\bar{r} = 2n-2-r$, and $-c_1 = n-1-\bar{r}$, whence (C2) is again satisfied. In this situation, $\lambda_0^r+ \rho_n$ can be transformed to $\rho_n$ as follows. Apply $s_{\varepsilon_1+\varepsilon_n}$ to get
$(0,\; n-1,\; n-2,\;\dots,\; n-\bar r,\; n-\bar r-2,\;\dots,\; 1,\; n-\bar r-1)$; then use adjacent reflections $s_{\varepsilon_i-\varepsilon_{i+1}}$: move $0$ from the first to the last position (requiring $n-1$ reflections), and move $n-\bar r-1$ left across the block $(n-\bar r-2,\dots,1)$ (requiring $n-\bar r-2$ reflections). Since each reflection contributes a factor $-1$, we have 
\[\mathbf{A}_n (e^{\lambda^{r}_{0}+\rho_n})
= (-1)^{2n-\bar{r}-2}\mathbf{A}_n(e^{\rho_n})= (-1)^{r}\mathbf{A}_n(e^{\rho_n}),
\qquad k=0 \text{ and  } n\le r\le 2n-2.\] 

{\bf Subcase 2.2:} $k\geq 1$. If $ r\leq n-1$, then $\bar{r}=r$. In this case, (C2) cannot hold, since $\pm c_1=n-1-\bar{r}$  would force $k=0$,  contradicting the assumption $k\geq 1$. Thus, we have
\[
\mathbf{A}_n(e^{\lambda_k^r+ \rho_n})=0, \quad 1\leq k\leq r\leq n-1.
\]
If $n\leq  r\leq 2n-2$, then $\bar{r} = 2n-2-r$.  The condition $-c_1=n-1-\bar{r}$ leads to $k=0$, which contradicts the assumption $k\geq 1$. The condition $c_1=n-1-\bar{r}$ leads to $r=n-1+\frac{k}{2}$, which forces $k$ to be an even integer. In this situation, $\lambda_k^r+\rho_n$ is carried to $\rho_n$ by moving the first coordinate past exactly $\bar{r}$ entries via the adjacent simple reflections $s_{\varepsilon_i-\varepsilon_{i+1}}, 1\leq i\leq \bar{r}$; each contributes a factor $-1$.

 Thus, for $1\leq k\leq r\leq 2n-2$, (C2) holds if and only if $k$ is even and $r=n-1+\frac{k}{2}$. In this case, we have 
\[ \mathbf{A}_n(e^{\lambda_k^r+ \rho_n})= (-1)^{\bar{r}}\mathbf{A}_n(e^{\rho_n})= (-1)^{n-1-\frac{k}{2}}\mathbf{A}_n(e^{\rho_n}). \]

Summarising the above case analysis according to the range of $k$, and applying the Weyl denominator formula $\Delta_n = \mathbf{A}_n(e^{\rho_n})$ and the Weyl character formula $\chi(\lambda_k^r)= \mathbf{A}_n  (e^{\lambda_k^r + \rho_n})/ \Delta_n$, we obtain the three stated expressions for $\mathbf{Ch}\,G_{n,k}$ in terms of characters.
\end{proof}

\section{Explicit generators of the centre}\label{sec: centre}

As an application of the representation-theoretic results, we will give new explicit generators of the centre of the quantum groups of types $\mathsf{B}$, $\mathsf{C}$, and $\mathsf{D}$ in terms of higher-order Casimir elements.

\subsection{Main result}
We first collect the notation and results we will use in the statement of the main result. The proof of the main result will be postponed to \secref{sec: proof}.

For any $\ell\geq 1$, recall from \secref{sec: highCas} that $C^0_{n,\ell}=\varphi(C_{n,\ell})$ denotes the image of the higher-order Casimir element $C_{n,\ell}$ under the Harish-Chandra isomorphism $\varphi$. By \corref{coro: C^0_nl},  the element $C^0_{n,\ell}$ can be expressed as
\[
C^0_{n,\ell}= \frac{1}{(q^{-1}-q)^{\ell}} \sum_{k=0}^{\ell} \binom{\ell}{k} (-q^{1-c_n})^{k}G_{n,k} \in (\mathrm{U}_{ev}^0)^W. 
 \]
Thus, Theorems \ref{thm: ChGtypeB}, \ref{thm: ChGtypeC} and \ref{thm: ChGtypeD} together provide a representation-theoretic interpretation of $C^0_{n,\ell}$ for all $\ell\geq 1$ in types $\mathsf{B}$, $\mathsf{C}$, and $\mathsf{D}$. 

On the other hand, \thmref{thm: generators of Uev0} shows that the $W$-invariant algebra $(\mathrm{U}_{ev}^0)^W$ is a polynomial algebra generated by the quantum Casimir elements  $C^0_{\varpi_1}, C^0_{\varpi_2},\dots, C^0_{\varpi_n}$, where $C^0_{\varpi_i}= \varphi(C_{\varpi_i})$ is the Harish-Chandra image of the   quantum Casimir element $C_{\varpi_i}$ of order one associated to the fundamental weight $\varpi_i$ for $1\leq i\leq n$. 

It is natural to ask whether the generators $C^0_{\varpi_i}, 1\leq i\leq n,$ can be expressed in terms of the higher-order Casimir elements $C^0_{n,\ell}$. In that case, one would obtain an alternative generating set of $(\mathrm{U}_{ev}^0)^W$, and hence of the centre $\mathcal{Z}(\mathrm{U})$.
This has been proved for type $\mathsf{A}_n$ in \cite{Li10}, while our main result below shows that the same holds for type $\mathsf{C}_n (n \geq 3)$, but fails in types $\mathsf{B}_n (n \geq 2)$ and $\mathsf{D}_n (n \geq 4)$. 

\begin{theorem}\label{thm: CentreGen}
Let $\mathrm{U}_q(\mathfrak{g})$ be the quantum group associated to the simple Lie algebra $\mathfrak{g}$ of type $\mathsf{B}_n\,(n\geq 2), \mathsf{C}_n\,(n\geq 3)$ or $\mathsf{D}_n\,(n\geq 4)$. For $1\leq \ell \leq n$,  let $C_{n,\ell}$ be the quantum Casimir elements of order $\ell$ associated with the natural representation,  and for $1\leq i\leq n$,  let $C_{\varpi_i}$ be the quantum Casimir elements of order one associated with the $i$-th fundamental weight $\varpi_i$ of $\mathfrak{g}$.
\begin{enumerate}
\item In type $\mathsf{B}_n\,(n\geq 2)$, the centre of $\mathrm{U}_q(\mathfrak{g})$ is generated by the quantum Casimir elements 
\[C_{n,1},\, C_{n,2},\dots,C_{n,n-1},\,C_{\varpi_n}.\] 
\item In type $\mathsf{C}_n\,(n\geq 3)$, the centre of $\mathrm{U}_q(\mathfrak{g})$ is generated by the quantum Casimir elements 
\[C_{n,1},\, C_{n,2}, \dots, C_{n,n}.\]

\item In type $\mathsf{D}_n\,(n\geq 4)$, the centre of $\mathrm{U}_q(\mathfrak{g})$ is generated by the quantum Casimir elements 
\[C_{n,1},\,\dots, C_{n,n-2},\, C_{\varpi_{n-1}}, \, C_{\varpi_n}.\]
\end{enumerate}
\end{theorem}

\begin{remark}
In type $\mathsf{B}_n\,(n\geq 2)$, the fundamental weight $\varpi_n= \frac{1}{2}(\varepsilon_1+ \cdots +\varepsilon_n)$ is the highest weight of the spin representation. In type $D_{n} (n\geq 4)$, there are two half-spin representations with highest weights $\varpi_{n-1}=\frac{1}{2}(\varepsilon_1+ \dots + \varepsilon_{n-1}- \varepsilon_n)$ and $\varpi_{n}=\frac{1}{2}(\varepsilon_1+ \dots + \varepsilon_{n-1}+\varepsilon_n)$.   
\end{remark}

\subsection{Character ring}

 Recall that the weight lattice $P$ of $\mathfrak{g}$ is a free abelian group generated by the fundamental weights $\varpi_1,\varpi_2,\dots,\varpi_n$. Denote by $\mathbb{K}[P]$ the group algebra of $P$ over $\mathbb{K}$, i.e., the Laurent polynomial algebra in the formal exponentials $e^{\lambda}$ with $\lambda\in P$. The Weyl group $W$ of $\mathfrak{g}$ acts naturally on $\mathbb{K}[P]$ by $w(e^{\lambda})=e^{w(\lambda)}$ for all $w\in W$ and $\lambda\in P$. The invariant subalgebra $\mathbb{K}[P]^W$ is called the \emph{character ring} of $\mathfrak{g}$ over $\mathbb{K}$.  

Let $2P\subset P$ denote the subgroup of even weights, i.e., all elements of the form  $2\lambda$ with $\lambda\in P$. Since $P\cong 2P$ as free abelian groups, the group algebras  $\mathbb{K}[P]$ and $\mathbb{K}[2P]$ are isomorphic as associative algebras via  $e^{\lambda}\mapsto e^{2\lambda}$ for $\lambda\in P$. Moreover, the subalgebra $\mathrm{U}_{ev}^0$ is naturally isomorphic to $\mathbb{K}[2P]$ through the correspondence $K_{2\lambda}\mapsto e^{2\lambda}$ for all $\lambda\in P$. All of these  algebra isomorphisms are $W$-equivariant, and they fit into the following commutative diagram, which identifies  the centre $\mathcal{Z}(\mathrm{U})$ with  the character ring $\mathbb{K}[P]^W$.
\[
\xymatrix{
\mathcal{Z}(\mathrm{U}) \ar[r]^{\varphi} &
(\mathrm{U}^0_{ev})^{W} \ar[r]^{\cong} \ar[dr]_{\cong} &
\mathbb{K}[2P]^{W} \ar[d]^{\cong} \\
&& \mathbb{K}[P]^{W}
}
\]

\begin{proposition}\cite[\S 23.2]{FH13}\label{prop: charring}
  Let $\mathfrak{g}$ be a simple Lie algebra of type $\mathsf{B}_n\,(n\geq 2), \mathsf{C}_n\,(n\geq 3)$ or $\mathsf{D}_n\,(n\geq 4)$. The character ring $\mathbb{K}[P]^W$ is a polynomial algebra generated by the characters $\chi(\varpi_1), \chi(\varpi_2), \dots, \chi(\varpi_n)$ of the fundamental representations of $\mathfrak{g}$. 
\end{proposition}

For $\mathfrak{g}$ of type $\mathsf{B}_n\,(n\geq 2), \mathsf{C}_n\,(n\geq 3)$ or $\mathsf{D}_n\,(n\geq 4)$, let
\[e_r:=\mathbf{Ch}(\bigw^rV), \quad 0\leq r\leq d:=\mathrm{dim}(V).\]
 denote the character of the $r$-th exterior power of the natural $\mathfrak{g}$-module $V$. By convention, $e_r=0$ for $r<0$. For $r>n$, we have $e_{r}=e_{d-r}$. These elementary symmetric polynomials are encoded in the generating function from \lemref{lem: genfun}. It follows from \lemref{lem: extpower} that 
\begin{align*}
 &\text{Type $\mathsf{B}_n\,(n\geq 2)$:} \quad  \chi(\varpi_r)= e_r, \quad 1\leq r\leq n-1,\\ 
  &\text{Type $\mathsf{C}_n\,(n\geq 3)$:} \quad  \chi(\varpi_1)=e_1,\quad  \chi(\varpi_2)=e_2-1,\quad   \chi(\varpi_r)= e_r-e_{r-2}, \quad 3\leq r\leq n,\\
 &\text{Type $\mathsf{D}_n\,(n\geq 4)$:} \quad   \chi(\varpi_r)= e_r, \quad 1\leq r\leq n-2.
\end{align*}
In particular, for type $\mathsf{C}_n\,(n\geq 3)$, the character ring $\mathbb{K}[P]^W$ is generated by the elementary symmetric polynomials $e_r$ for $1\leq r\leq n$. For types $\mathsf{B}_n\,(n\geq 2)$ and $\mathsf{D}_n\,(n\geq 4)$,  spin and half-spin characters cannot be expressed in terms of the corresponding  elementary symmetric polynomials \cite[\S 23.2]{FH13}.

More generally, the irreducible characters $\chi(\lambda)$ of $\mathfrak{g}$ indexed by partitions $\lambda$ can be expressed in terms of the elementary symmetric polynomials $e_r$ via the Jacobi-Trudi-type formulae \cite{FH13}, which we now recall. A \emph{partition} is a finite non-increasing sequence of positive integers $\lambda=(\lambda_1,\lambda_2,\dots,\lambda_k)$.  The Young diagram of shape $\lambda$ is an array of left-justified boxes with $\lambda_i$ boxes in the $i$-th row. The \emph{conjugate partition} $\lambda'=(\lambda'_1,\lambda'_2,\dots,\lambda'_{\ell})$ is defined by letting $\lambda'_j$ be the number of boxes in the $j$-th column of the Young diagram of shape $\lambda$. A partition $(\lambda_1,\lambda_2,\dots,\lambda_k)$ with $k\leq n$ determines  a dominant integral weight  $\lambda=\sum_{i=1}^k \lambda_i \varepsilon_i$ of $\mathfrak{g}$.

The following is the well known Jacobi-Trudi identity in terms of elementary symmetric polynomials for  types $\mathsf{B}, \mathsf{C}$ and $\mathsf{D}$ \cite{Kin71} (see also \cite[\S 24.2]{FH13}). 

\begin{proposition} \cite{Kin71}
 Let $\lambda=(\lambda_1, \dots, \lambda_k)$ be a partition with $k\leq n$, and let $\lambda'= (\lambda'_1, \dots, \lambda'_{\ell})$ be the conjugate partition to $\lambda$, where $\ell=\lambda_1$. Let $\mathrm{det}(m_{ij})_{1\leq i,j\leq \ell}$ denotes the determinant of the matrix $(m_{ij})_{1\leq i,j\leq \ell}$.
 \begin{enumerate}
  \item In type $\mathsf{B}_n\,(n\geq 2)$ or $\mathsf{D}_n\, (n\geq 4)$, we have 
  \[ \chi(\lambda)= \frac{1}{2} \mathrm{det}(e_{\lambda'_i-i+j}+ e_{\lambda'_i-i-j+2})_{1\leq i,j\leq \ell}  \]
  \item In type $\mathsf{C}_n\,(n\geq 3)$, we have 
  \[\chi(\lambda)= \mathrm{det}(e_{\lambda'_i-i+j}- e_{\lambda'_i-i-j})_{1\leq i, j\leq \ell}. \]
  \end{enumerate} 
\end{proposition}

Straightforward application to hook partitions yields the following. 

\begin{corollary}\label{coro: hookform}
   Let $\lambda_k^r=(k-r, 1^r)$ be a hook partition (\figref{fig: youngdia}) with $k-r\geq 1$ and $0\leq r\leq n-1$.  Then in type $\mathsf{B}_n\,(n\geq 2)$ or $\mathsf{D}_n\,(n\geq 4)$, we have 
 \[\chi(\lambda_k^r)=\mathrm{det}
\begin{pmatrix}
e_{r+1} & e_{r+2}+e_r & e_{r+3}+e_{r-1} & \cdots & e_k+e_{2r-k+2} \\[6pt]
1       & e_1         & e_2             & \cdots & e_{k-r-1} \\[6pt]
0       & 1           & e_1             & \cdots & e_{k-r-2} \\[6pt]
\vdots  & \vdots      & \vdots          & \ddots & \vdots \\[6pt]
0       & 0           & 0               & \cdots\quad  1 & e_1
\end{pmatrix},
\]
and in type $\mathsf{C}_n\,(n\geq 3)$, we have 
 \[\chi(\lambda_k^r)=\mathrm{det}
\begin{pmatrix}
e_{r+1}-e_{r-1} & e_{r+2}-e_{r-2} & e_{r+3}-e_{r-3} & \cdots & e_k-e_{2r-k} \\[6pt]
1       & e_1         & e_2             & \cdots & e_{k-r-1} \\[6pt]
0       & 1           & e_1             & \cdots & e_{k-r-2} \\[6pt]
\vdots  & \vdots      & \vdots          & \ddots & \vdots \\[6pt]
0       & 0           & 0               & \cdots \ 1 & e_1
\end{pmatrix}.
\]

\end{corollary}

\subsection{Proof of the main result}\label{sec: proof}

The following is a direct consequence of Theorems \ref{thm: ChGtypeB}, \ref{thm: ChGtypeC} and \ref{thm: ChGtypeD}.
\begin{lemma}\label{lem: ChGexp}
 Maintain notation as in Theorems \ref{thm: ChGtypeB}, \ref{thm: ChGtypeC} and \ref{thm: ChGtypeD}. Let $\delta_{k,\, ev}$ be $1$ if $k$ is even and $0$ otherwise.  For all integers $1\leq k\leq n$, we have 
\begin{align*}
\text{Type $\mathsf{B}_n\,(n\geq 2)$:} \quad &\mathbf{Ch}\, G_{n,k} = \delta_{k,\, ev}\, q^{-k} + \sum_{r=0}^{k-1} (-1)^r q^{2n - 2r - 1} \chi(\lambda_k^{r}),\\
\text{Type $\mathsf{C}_n\,(n\geq 3)$:} \quad &\mathbf{Ch}\, G_{n,k} = -\delta_{k,\, ev}\, q^{-k} + \sum_{r=0}^{k-1} (-1)^r q^{2n - 2r} \chi(\lambda_k^{r}),\\
\text{Type $\mathsf{D}_n\,(n\geq 4)$:} \quad &\mathbf{Ch}\, G_{n,k} = \delta_{k,\, ev}\, q^{-k} + \sum_{r=0}^{k-1} (-1)^r q^{2n - 2 - 2r} \Big(\chi(\lambda_k^{r}) + \delta_{r,n-1} \chi(\bar{\lambda}_k^{n-1})\Big).
\end{align*}
\end{lemma}

\begin{proposition}\label{prop: algindep}
In type $\mathsf{B}_n\, (n\geq 2), \mathsf{C}_n\,(n\geq 3)$ or $\mathsf{D}_n\,(n\geq 4)$, for each $k=1, \dots ,n$, there exist a nonzero  $c_k\in \mathbb{K}$ and a polynomial $Q_{k}\in \mathbb{K}[t_1, \dots , t_{k-1}]$ such that 
\[e_{k}= c_k g_k + Q_k(g_1, \dots, g_{k-1}),\]
where $g_r:= {\bf Ch}\, G_{n,r}$ and $e_{r}= {\bf Ch}(\bigw^r V)$ for $1\leq r\leq n$. Moreover, the elements $g_1, \dots, g_n$ are algebraically independent over $\mathbb{K}$. Therefore, we have $\mathbb{K}[g_1, \dots, g_n]= \mathbb{K}[e_1, \dots, e_n]$. 
\end{proposition}
\begin{proof}
We treat type $\mathsf{B}_n$ first. The proof proceeds by induction on $k$.  For $k=1$, by \lemref{lem: ChGexp},  we have $g_1= q^{2n-1}\chi(\lambda_{1}^0)=q^{2n-1}e_1$, so $c_1= q^{1-2n}$ and $Q_1=0$. 
For $k>1$, using \corref{coro: hookform} and expanding the determinant along the first row, we have
\[ \chi(\lambda_k^r)=(-1)^{k-r+1} e_k + Q_{k,r}(e_1, \dots, e_{k-1}), \quad 0\leq r\leq k-1,  \]
where each $Q_{k,r}\in \mathbb{Z}[t_1, \dots, t_{k-1}]$. Substituting into the formula  in \lemref{lem: ChGexp},  we have 
\begin{align*}
  g_k&=\delta_{k, \, ev}q^{-k}+ \sum_{r=0}^{k-1} (-1)^r q^{2n - 2r - 1} \chi(\lambda_k^{r})\\ 
    &=\delta_{k, \, ev}q^{-k}+ (-1)^{k+1}e_k \sum_{r=0}^{k-1} q^{2n - 2r - 1}+ \sum_{r=0}^{k-1}(-1)^r q^{2n-2r-1} Q_{k,r}(e_1, \dots, e_{k-1}). 
\end{align*}
Let $c_{k}= (-1)^{k+1}(\sum_{r=0}^{k-1}q^{2n-2r-1})^{-1}$. Then we solve for $e_k$ and obtain 
\[
    e_{k}= c_kg_k - \delta_{k,\, ev} q^{-k}c_k - c_k\sum_{r=0}^{k-1}(-1)^r q^{2n-2r-1} Q_{k,r}(e_1, \dots, e_{k-1})
\]
 By the induction hypothesis, each $e_j$ with $j<k$ lies in $\mathbb{K}[g_1, \dots, g_{k}]$, we have $e_k= c_kg_k + Q_k(g_1, \dots, g_{k-1})$ for some polynomial $Q_k\in \mathbb{K}[t_1, \dots, t_{k-1}]$.

 Moreover, it is well known that $e_1, \dots, e_n$ are algebraically independent. Since each $g_k$ is of the form $g_k= c_k^{-1} e_k + R_k(e_1, \dots, e_{k-1})$ with $c_k\neq 0$ for some polynomial $R_k$, the Jacobian determinant $\mathrm{det}(\partial_{e_j}g_i)_{1\leq i,j \leq n}= c^{-1}_1\cdots c^{-1}_n$ is nonzero, so the elements $g_1, \dots, g_n$ are algebraically independent (see, e.g., \cite[Theorem 11.4]{Lef53}). Therefore, we have $\mathbb{K}[g_1, \dots , g_n]=\mathbb{K}[e_1, \dots, e_n]$.

The same argument applies to types $\mathsf{C}$ and $\mathsf{D}$ with minor changes.  The only additional case to handle is  in type $\mathsf{D}$ with  $k=n$ and $r=n-1$. In this case, it follows from  \lemref{lem: extpower} that 
\[
e_n= \chi(\lambda_{n}^{n-1})+ \chi(\bar{\lambda}_{n}^{n-1}). 
\]
Substituting into the formula in \lemref{lem: ChGexp},  we obtain 
\[
    g_n=  \delta_{n,\, ev}\, q^{-n} + \sum_{r=0}^{n-2} (-1)^r q^{2n - 2 - 2r} \chi(\lambda_n^{r}) + (-1)^{n-1}e_n. \]
As before, each $\chi(\lambda_n^r)$ for $r \leq n - 2$ is expressible as a polynomial in $e_1,\dots,e_{n}$, hence we can solve for $e_n$ and the coefficient of $g_n$ is nonzero. The rest  proceeds as in type $\mathsf{B}_n$.
\end{proof}

We are in position to prove the main result. 

\begin{proof}[Proof of \thmref{thm: CentreGen}]
By the Harish-Chandra isomorphism, it is equivalent to proving the corresponding statements for the $W$-invariant algebra $(\mathrm{U}_{ev}^0)^W$, which is isomorphic to the character ring $\mathbb{K}[P]^W$ via sending  $C^0_{\varpi_i}$ to $\chi(\varpi_i)$ for all $1\leq i\leq n$. By \propref{prop: charring}, $\mathbb{K}[P]^W$ is a polynomial algebra generated by the characters $\chi(\varpi_1), \chi(\varpi_2), \dots, \chi(\varpi_n)$. 

By \corref{coro: C^0_nl},  the two families  $C^0_{n,1}, \dots, C^0_{n,n} $  and $G_{n,1}, \dots, G_{n,n}$  generate the same $\mathbb{K}$-subalgebra of $(\mathrm U^0_{ev})^W$. Passing to the character ring via $\mathbf{Ch}$, \propref{prop: algindep} yields
\[\mathbb{K}[g_1,\dots, g_n] =\mathbb K[e_1,\dots,e_n]\subseteq \mathbb K[P]^W\]
 where $g_k=\mathbf{Ch}(G_{n,k})$  and  $e_k=\mathbf{Ch}(\bigw^k V)$ for the natural module $V$ for $1\leq k\leq n$.

We now conclude case by case, using the algebraic structure of $\mathbb{K}[P]^W$ in each type. 

Type $\mathsf{B}_n\, (n\geq 2)$: By \propref{prop: charring} and \lemref{lem: extpower}, one has \[\mathbb{K}[P]^W=\mathbb{K}[e_1,\dots,e_{n-1},\chi(\varpi_n)].\]
 In view of \propref{prop: algindep}, we have $\mathbb{K}[e_1,\dots,e_{n-1}]=\mathbb{K}[g_1, \dots, g_{n-1}]$. It follows that $(\mathrm U^0_{ev})^W=\mathbb{K}[C^0_{n,1},\dots,C^0_{n,n-1},\,C^0_{\varpi_n}].$ Pulling back along the Harish-Chandra isomorphism, the central elements $C_{n,1},\dots,C_{n,n-1}, C_{\varpi_{n}}$ generate $\mathcal{Z}(\mathrm{U}_q(\mathfrak g))$.

Type $\mathsf{C}_n\,(n\geq 3)$: By \propref{prop: charring} and \lemref{lem: extpower},   one has $\mathbb{K}[P]^W=\mathbb{K}[e_1,\dots,e_n]$. Therefore the images $C^0_{n,1},\dots,C^0_{n,n}$ generate $(\mathrm U^0_{ev})^W$. Pulling back along the Harish-Chandra isomorphism, the central elements $C_{n,1},\dots,C_{n,n}$ generate $\mathcal{Z}(\mathrm{U}_q(\mathfrak g))$. 

Type $\mathsf{D}_n\,(n\geq 4)$: By \propref{prop: charring} and \lemref{lem: extpower}, one has \[\mathbb{K}[P]^W=\mathbb{K}[e_1,\dots,e_{n-2},\chi(\varpi_{n-1}), \chi(\varpi_n)].\]
By  \propref{prop: algindep}, we have $ \mathbb{K}[e_1,\dots,e_{n-2}]=\mathbb{K}[g_1, \dots, g_{n-2}]$. It follows that 
$(\mathrm U^0_{ev})^W=\mathbb{K}[C^0_{n,1},\dots,C^0_{n,n-2},\,C^0_{\varpi_{n-1}},\,C^0_{\varpi_n}]$,
and pulling back along the Harish-Chandra isomorphism yields the  central generators $C_{n,1},\dots,C_{n,n-2},C_{\varpi_{n-1}},C_{\varpi_n}$.
\end{proof}

\begin{remark}
 In type $\mathsf{A}_n\ (n\geq 1)$, the centre of $\mathrm{U}_q(\mathfrak{sl}_{n+1})$ is generated by the higher-order quantum Casimir elements $C_{n,1}, C_{n,2}, \dots, C_{n,n}$ associated with the natural representation \cite{Li10}.  This follows by the same reasoning as in \lemref{lem: ChGexp} and \propref{prop: algindep}. In fact, \cite[Lemma 3.3]{Li10} is the type~$\mathsf{A}$ analogue of \lemref{lem: ChGexp} and is proved by a more direct method that does not involve the auxiliary elements $H_{n,k}$. Nonetheless, the same result can also be recovered from our approach, using the elements $H_{n,k}$ and following the method in the proof of \thmref{thm: ChGtypeB}.
\end{remark}

We conclude this paper with a remark on \lemref{lem: ChGexp}. 

\begin{remark}\label{rem: stab}
Fix an integer $k\geq 1$. In types $\mathsf{B}_n$ and $\mathsf{C}_n$, for all $n\geq k$, the set of irreducible constitutes of $\mathbf{Ch}\,G_{n,k}$ is independent of $n$: namely, the trivial character (present when $k$ is even) together with  $\chi(\lambda_k^r)$, where
\[
\lambda_k^r= (k-r)\varepsilon_1+ \varepsilon_2+ \cdots + \varepsilon_{r+1}, \quad 0\leq r\leq k-1,
\]
each appearing with multiplicity one.  In type $\mathsf{D}_n$, the same stability  holds for all $n>k$. This may be viewed as a quantum analogue of the  representation stability phenomenon in the sense of Church and Farb \cite{CF13}, with the caveat that the powers of $q$ accompanying the irreducible constituents  of $\mathbf{Ch}\,G_{n,k}$ do not yet have a natural conceptual interpretation.
\end{remark}

\end{document}